\documentclass[12pt,a4paper]{scrartcl}
\usepackage[T1]{fontenc}
\usepackage{mathpazo}
\usepackage[english]{babel}
\usepackage{german}
\usepackage{csquotes}

\usepackage{graphicx}
\usepackage[a4paper]{geometry}
\usepackage{fullpage}
\usepackage{hyperref}
\usepackage{algorithm}
\usepackage{tikz}
\usepackage[arrow, matrix, curve]{xy}
\usepackage{setspace}
\usepackage{tabto}

\usepackage{amsmath}
\usepackage{amsthm}
\usepackage{amssymb}
\usepackage[mathscr]{eucal}

\theoremstyle{definition}
\newtheorem{defi}{Definition}[section]
\newtheorem*{rema}{Remark}
\newtheorem{conj}[defi]{Conjecture}

\theoremstyle{plain}
\newtheorem{theo}[defi]{Theorem}
\newtheorem{prop}[defi]{Proposition}
\newtheorem{coro}[defi]{Corollary}
\newtheorem{lemm}[defi]{Lemma}

\numberwithin{equation}{section}

\DeclareMathOperator{\e}{e}
\DeclareMathOperator{\nd}{d}
\DeclareMathOperator{\ld}{ld}

\DeclareMathOperator{\rad}{rad}

\DeclareMathOperator*{\lcm}{lcm}

\newcommand{\exhead}[1]{\ \\\textit{#1.}\vspace*{-8mm}}
\newcommand{\phj}{\phantom{j\hspace{-1mm}}}
\newcommand{\refs}[1]{\ref{#1}\,}
\def\lb{\linebreak}

\newcommand{\mD}{\mathbb{D}}
\newcommand{\mM}{\mathbb{M}}
\newcommand{\mN}{\mathbb{N}}
\newcommand{\mZ}{\mathbb{Z}}
\newcommand{\mP}{\mathbb{P}}
\newcommand{\mR}{\mathbb{R}}
\newcommand{\mC}{\mathbb{C}}
\newcommand{\mS}{\mathbb{S}}

\newcommand{\msA}{\mathscr{A}}
\newcommand{\msD}{\mathscr{D}}
\newcommand{\msI}{\mathscr{I}}
\newcommand{\msM}{\mathscr{M}}
\newcommand{\msG}{\mathscr{G}}

\newcommand\around[1]{\,\tikz[baseline]\node[draw,shape=rectangle,rounded  corners,scale=0.5,anchor=-5em] {#1} ;}

\newcommand{\mdot}{\!\cdot\!}
\newcommand{\mplus}{\!+\!}
\newcommand{\ndiv}{\!\nmid\!}
\newcommand{\mdiv}{\!\mid\!}
\newcommand{\copr}{\!\perp\!}

\def\qs{\,/\!}

\newcommand{\fcd}[3]{\frac{#1}{\gcd(#2,#3)}}
\newcommand{\fcn}[3]{\frac{\gcd(#1,n)}{\gcd(#2,#3,n)}}
\newcommand{\pcd}[6]{\frac{#1\cdot#2}{\gcd(#3,#4)\cdot\gcd(#5,#6)}}

\newcommand{\tabalign}{\hspace{20mm}&\hspace{140mm}\\}
\newcommand{\tb}{\hspace{5mm}}

\begin{document}

\title{Orders and partitions of integers induced by arithmetic functions}
\author{Mario Ziller}
\date{}

\maketitle

\begin{abstract}
We pursue the question how integers can be ordered or partitioned\lb according to their divisibility properties. Based on pseudometrics on $\mZ$,\lb we investigate induced preorders, associated equivalence relations, and\lb quotient sets. The focus is on metrics or pseudometrics on $\mD_n$, the set of divisors of a given modulus $n\in\mN$, that can be extended to pseudometrics on $\mZ$.

Arithmetic functions can be used to generate such pseudometrics. We discuss several subsets of additive and multiplicative arithmetic functions and various combinations of their function values leading to binary metric functions that represent different divisibility properties of integers. 

We conclude this paper with numerous examples and review the most important results. As an additional result, we derive a necessary condition for the truth of the odd k-perfect number conjecture.
\end{abstract}

\tableofcontents

\section*{Introduction}
\addcontentsline{toc}{section}{\phj Introduction\phj}

Henceforth, we denote the set of integral numbers by $\mZ$ and the set of natural\lb numbers, i.e. positive integers, by $\mN$. $\mN_0=\mN\cup\{0\}$. $\mP=\{p_i : i\in\mN\}$ is the set of prime numbers with $p_1=2$. 

The set of divisors of $n$, we abbreviate by $\mD_n=\{x\in\mN : x\mdiv n\}$, whereas the set of prime powers dividing $n$ is $\mP^*_n=\bigl\{x\in\mD_n:x=p^{\,i},p\in\mP,i\in\mN\bigr\}$.

The set of real numbers is denoted by $\mR$, and $\mR_{\ge0}=\{x\in\mR: x\ge0\}$.\\[-1.75mm]

Natural numbers are known to form lattices in at least two ways \cite{Davey_Priestley_2002, Burris_Sankappanavar_2012}. The usual\lb order relation \enquote{$\le$} is a total order on $\mN$, and $\{\mN,\min,\max\}$ is a lattice which\lb elucidates the additive characteristic of $\mN$. On the other hand, the lattice $\{\mN,\gcd,\lcm\}$ emphasises the multiplicative structure of $\mN$ which is partially ordered by the\lb divisibility relation \enquote{ $\mdiv$ }. While \enquote{$\le$} is also a total order on $\mZ$, \enquote{ $\mdiv$ } is just a preorder on $\mZ$.

For every $n\in\mN$, the set of divisors of $n$ forms a sublattice of $\{\mN,\gcd,\lcm\}$ with the same operations $\gcd$ and $\lcm$ restricted to $\mD_n$. All sets of divisors $\mD_n$ are even partially ordered by \enquote{ $\mdiv$ } because $\mD_n\subset\mN$.  In this paper, we focus on divisibility properties of integers as a whole. We pursue the question which further partial orders or preorders\lb on $\mZ$ and $\mD_n$ could provide additional insight into the divisibility properties of\lb integers.\\[-1.75mm]

We examine specific metrics and pseudometrics on $\mD_n$ or $\mZ$ and their corresponding induced partial orders and preorders in Section~\refs{MO}. Moreover, we point to an interesting relationship between metrics and orders. For each pseudometric space, a preorder on this space and an associated pre-ordering function can be generated by defining a central point.

In particular, we demonstrate that a preorder on the set of integers $\mZ$ can be\lb constructed from a pseudometric in the following way. Let $x,y\in\mZ$. For any pseudo-\lb metric $d$ on $\mZ$, the relation $x\preceq y \iff d(1,x)\le d(1,y)$ is a preorder on $\mZ$. The function $f(x)=d(1,x)$ is the related pre-ordering function. Solely comparing the\lb distances to $1$ should not impede studying the divisibility properties, since $1$ divides all integers. On the other hand, any non-negative function $f:\mN\to\mR_{\ge0}$ with $f(1)=0$ can be used to generate a centred pseudometric on $\mZ$. This motivates us to examine special arithmetic functions in this context in Section~\refs{AFM}.

Pseudometrics and preorders on the set of divisors of $n$ play a central role\lb in studying the divisibility properties with respect to a given modulus $n\in\mN$. We show that every pseudometric and its induced preorder on $\mD_n$ can be extended to a pseudometric and preorder on $\mZ$, respectively. These extensions are explored in later sections.\\[-1.75mm]

Based on the considered pseudometrics and preorders, we derive associated\lb equivalence relations and quotient sets in Section~\refs{EP}. By this means, partitions of\lb integers can be generated with respect to divisibility properties. Partitions of $\mD_n$ can be extended to partitions of $\mZ$, similar to the extension of pseudometrics and\lb preorders.

Partitions can be formed such that one is a refinement of another. We demonstrate that the values of the pre-ordering function or of the pseudometric distance essentially control the fineness of partitions.\\[-1.75mm]

In Section~\refs{AFM}, we explore pseudometrics on $\mD_n$ or $\mZ$ that are generated by arithmetic functions. We discuss several subsets of additive or multiplicative arithmetic functions which can be used for this purpose.

Certain sets of admissible functions together with pointwise sum and product,\lb respectively, form commutative monoids or semigroups. They are isomorphic with respect to both, the binary operations and their additive or multiplicative properties.

We present various combinations of function values leading to binary metric\lb functions that represent different divisibility properties of integers. A further\lb investigation of these divisibility metrics is beyond the scope of this paper and should be continued elsewhere.\\[-1.75mm]

As generally described in the first two sections, in Section~\ref{AFO} we examine orders, relations and partitions of $\mD_n$ or $\mZ$ generated from admissible arithmetic functions.

We derive some specific types of preorders on $\mZ$ and show that the corresponding quotient sets are invariant against permutation of the values of the respective function on $\mD_n$. The function values on $\mP^*_n$ can be chosen arbitrarily within the defined range. This can lead to a lot of combinatorial possibilities and thus to different structures of equivalence classes. We discuss several special cases.\\[-1.75mm]

We conclude this article with numerous examples and review the most important results.

\ \\\\[-1.75mm]


\section{Metrics and orders\phj} \label{MO}

In the following section, we investigate some general relationships between metrics and orders on sets of integral numbers. We describe the context for arbitrary sets of integers $\mM\subseteq\mZ$ if applicable. These include e.g. $\mD_n$, $\mN$, and $\mZ$ itself. Furthermore, corresponding equivalence classes and quotient sets will be characterised.\\

To begin with, we recall some important, well-known terms for comparison and\lb restrict the definitions to sets of integers. Let $x,y,z \in \mM$. For further reading, we refer to textbooks, e.g. \cite{Schechter_1997, OSearcoid_2006, Shirali_Vasudeva_2006} and \cite{Davey_Priestley_2002, Burris_Sankappanavar_2012}.\\

A function $d \colon \mM\times\mM \to \mR$ is called \textbf{pseudometric} on $\mM$ if the following conditions hold: \vspace*{-2mm}
\TabPositions{5cm, 8cm, 11cm}
\begin{itemize}\itemsep-6mm
\item $d(x,x)=0$ \tab identity,\\
\item $d(x,y)=d(y,x)$	 \tab symmetry,\\
\item $d(x,z)\leq d(x,y)+d(y,z)$ \tab triangle inequality.
\end{itemize} \newpage
We deduce $d(x,y)\!\ge\!0$ from $0=d(x,x)\leq d(x,y)+d(y,x)=2 \mdot d(x,y)$.

A pseudometric that additionally meets
\vspace*{-2mm}
\begin{itemize}\itemsep-6mm
\item  $d(x,y)\!=\!0 \Longrightarrow x\!=\!y$ \tab identity of indiscernibles,
\end{itemize} \vspace*{-2mm}
is a \textbf{metric} on $\mM$. Then $d(x,y)\!=\!0\!\iff\!x\!=\!y$. In other words, in a pseudometric space there can even exist elements $x\!\neq\!y$ with $d(x,y)\!=\!0$.

Every pseudometric or metric on $\mM$ is also a pseudometric or metric, respectively, on any subset of $\mM$.\\[-1.75mm]

The \textbf{set of pseudometrics} on $\mM$ form a commutative, additive \textbf{monoid}. Let $d_1$, $d_2$ be pseudometrics on $\mM$, and $x,y\in\mM$. The binary operation \enquote{$\circ$} defined by
\vspace*{-3mm}
\[d_1\circ d_2(x,y)=d_1(x,y)+d_2(x,y)\]
\vspace*{-8mm}\\
is associative. The identity element is $d_0(x,y)=0$ for all $x,y\in\mM$. The set of pseudo-\lb metrics is closed under the binary operation because $d_1\circ d_2$ is also a pseudometric according to the axioms mentioned above.

Since $d_0$ is not a metric, the \textbf{set of metrics} on $\mM$ with \enquote{$\circ$} is only a commutative, additive \textbf{semigroup}, and consequently a commutative, additive subsemigroup of the set of pseudometrics.

Let $\lambda\in\mR>0$, and $d$ be a pseudometric or even metric on $\mM$. Then
\vspace*{-3mm}
\[(\lambda\cdot d)(x,y)=\lambda\cdot d(x,y)\]
\vspace*{-8mm}\\
is also a pseudometric or metric, respectively. This leads to \textbf{modules} over the Abelian group $(\mR^+,\cdot)$ where $\mR_+=\{x\in\mR: x>0 \}$.\\[-1.75mm]

A reflexive and transitive relation $\preceq$ on $\mM$, i.e. a subset of $\mM\times\mM$ with \vspace*{-2mm}
\TabPositions{5cm, 8cm, 11cm}
\begin{itemize} \itemsep-6mm
\item $x \preceq x$ \tab reflexivity \quad and\\
\item $x \preceq y \land y \preceq z \Longrightarrow x \preceq z$ \tab transitivity
\end{itemize} \vspace*{-2mm}
is called \textbf{preorder} or \textbf{quasiorder} on $\mM$.

A  \textbf{partial order} is an antisymmetric preorder, i.e.
\vspace*{-2mm}
\begin{itemize}\itemsep-6mm
\item  $x \preceq y \land y \preceq x \Longrightarrow x=y$ \tab antisymmetry
\end{itemize} \vspace*{-2mm}
also applies. In a preorder, however, there can exist elements $x\!\neq\!y$ with $x \preceq y$ and $y \preceq x$.

Two preorders $\preceq_1$ and $\preceq_2$ on $\mM$ are called \textbf{equivalent} if $x \preceq_1 y \iff x \preceq_2 y$ for\lb all $x,y\in\mM$.\\[6mm]

Based on the Euclidean metric on $\mR^2$, the classical French railway metric is one of the best-known examples of a centred metric. It was inspired by the former French railway network, in which almost all rail connections converged in Paris.

Its definition reads\\[-7mm]
\begin{align*}
d_F(x,y)=\begin{cases}
   d(x,y)\hspace{16.8mm}=\|y-x\| &\!\text{if $x$ and $y$ are collinear with $o$, i.e.}\\
   &\!\exists\ \lambda\in\mR:o=\lambda\mdot x+(1\!-\!\lambda)\mdot y,\text{ and}\\
   d(o,x)+d(o,y)=\|x-o\|+\|y-o\| &\!\text{otherwise,} \end{cases}
\end{align*}\\[-5mm]
where $o,x,y\in\mR^2$, $o$ is the central point, and $d(x,y)=\|y-x\|$ is the Euclidean norm.

If one dispenses with the condition of collinearity, the idea of a centred metric can be transferred even more easily to any pseudometric space \cite{Bulgarean_2013}. Let $\mS$ be a space and $d$ a pseudometric on $\mS$. Then for every fixed point $o\in\mS$ and arbitrary $x,y\in\mS$,
\begin{align*}
d_o(x,y)=\begin{cases}
   0 &\text{for $x=y$, \quad and}\\
   d(o,x)+d(o,y) &\text{for $x\ne y$} \end{cases}
\end{align*}
is also a pseudometric on $\mS$. The function $f_o:\mS\rightarrow\mR_{\ge0}$
with $f_o(x)=d(o,x)$\lb reconstructs $d_o$. On the other hand, any non-negative function $f_o:\mS\to\mR_{\ge0}$ with $f(o)=0$ can be used to generate an $o$-centred pseudometric on $\mS$ by means of
\begin{align*}
  d_o(x,y)&\ \ =\ \ \begin{cases}
    0 &\text{for $x=y$, \quad and}\\
    f_o(x)+f_o(y) &\text{for $x\ne y$.} \end{cases}
\end{align*}

The function $f_o$ generating an $o$-centred pseudometric $d_o$ on $\mS$ can also be used to pre-order $\mS$. The relation defined by
\[ x\preceq_o y \iff d(o,x)\le d(o,y) \]
is a preorder on $\mS$.\\\\

In the following, we focus on pseudometrics on sets of integers and aim to consider the divisibility properties of its elements. Since $1$ is the only natural number dividing all integers, it seems obvious to choose $1$ as the centre of a pseudometric and study corresponding preorders.

For the general case of subsets $\mM\subseteq\mZ$ of integers, we additionally require\lb $1\in\mM$. Then, every pseudometric $d$ on $\mM$ induces an associated function\lb $f_d(x)=d(1,x)$ defining a 1-\hspace*{0.25mm}centred pseudometric. Vice versa, every non-negative function $f_d:\mM\to\mR_{\ge0}$ with $f_d(1)=0$ generates a $1$-\hspace*{0.25mm}centred pseudometric on $\mM$.
\begin{lemm} \label{centred_pseudometric}
Let $1,x,y\in\mM$, and $d$ be a pseudometric on $\mM$. With the generating function $ f_d(x)=d(1,x)$,
\begin{align*}
d_1(x,y)=\begin{cases}
   0 &\text{for $x=y$, \quad and}\\
   f_d(x)+f_d(y)=d(1,x)+d(1,y) &\text{for $x\ne y$,} \end{cases}
\end{align*}
is also a pseudometric on $\mM$.
\end{lemm}
\begin{proof}
Identity and symmetry of $d_1$ are clear by definition. Given $x,y,z\in\mM$, the triangle inequality reads
\vspace*{-2mm}
\begin{align*}
d_1(x,z)&=f_d(x)+f_d(z)\\
&\le f_d(x)+f_d(y)+f_d(y)+f_d(z)=d_1(x,y)+d_1(y,z).
\end{align*}
\vspace*{-14.5mm}\\
\end{proof}
\ \\

Every $1$-\hspace*{0.25mm}centred pseudometric on $\mM$ is uniquely related to a non-negative function $f:\mM\to\mR_{\ge0}$ with $f(1)=0$. However, the same centred pseudometric can be induced by different pseudometrics because only the distances to $1$ were used to construct it.

In general, such functions can be used to generate 1-\hspace*{0.25mm}centred pseudometrics or even metrics on $\mM$.

\begin{lemm} \label{function_metric}
Let $x,y\in\mM$, and $f:\mM\to\mR_{\ge0}$ be a function with $f(1)=0$. Then 
\begin{align*}
d(x,y)=\begin{cases}
   0 &\text{for $x=y$, \quad and}\\
   f(x)+f(y) &\text{for $x\ne y$,} \end{cases}
\end{align*}
is a pseudometric on $\mM$.

Furthermore, if $f(x)>0$ for $x>1$ then $d$ is even a metric.
\end{lemm}
\begin{proof}
Identity, symmetry, and triangle inequality result similar to the proof of\linebreak
Lemma~\refs{centred_pseudometric}.

Let furthermore $f(x)>0$ for $x>1$ and $x\neq y$. W.l.o.g. we can assume $1\le x<y$. Then $d(x,y)=f(x)+f(y)\ge f(y)>0$. In other words, $d(x,y)=0 \Longrightarrow x=y$.
\end{proof}

This lemma motivates to a separate section. Functions that meet the requirements\linebreak of the previous lemma are special arithmetic functions which we will discuss in\linebreak Section~\refs{AFM}.\\\\

We now introduce a specific order relation in connection with a pseudometric. For any given pseudometric, this relation turns out to be a preorder. Under certain\lb circumstances, it can even be a partial order.

As above, we use generating functions for the description to emphasise the duality between centred pseudometrics and related generating functions.

\begin{lemm} \label{induced_preorder}
Let $1,x,y\in\mM$, $d$ be a pseudometric on $\mM$, and $f_d$ the generating function. Then, the relation $\preceq_d$ defined by
\vspace*{-3mm}
\[x\preceq_d y \iff f_d(x)\le f_d(y)\] \vspace*{-7mm}\\
is a preorder on $\mM$.

If the generating function $f_d$ also satisfies the condition $f_d(x)\!=\!f_d(y) \Longrightarrow x\!=\!y$\lb then $\preceq_d$ is even a partial order.
\end{lemm}
\begin{proof}
The specified relation is reflexive because $f_d(x)\le f_d(x)$, so $x\preceq_d x$. It is also transitive because the ordinary \enquote{is less than or equal to} relation is transitive on $\mR_{\ge0}$. If $x\preceq_d y$ and $y\preceq_d z$ then $f_d(x)\le f_d(y)$ and $f_d(y)\le f_d(z)$, and consequently\lb $f_d(x)\le f_d(z)$ , i.e. $x\preceq_d z$.

Moreover, $\preceq_d$ is antisymmetric if $f_d(x)\!=\!f_d(y) \Longrightarrow x\!=\!y$ because
\vspace*{-3mm}
\begin{align*}
x \preceq_d y \land y \preceq_d x &\Longrightarrow f_d(x)\le f_d(y) \land f_d(y)\le f_d(x)\\
&\Longrightarrow f_d(x)\!=\!f_d(y) \Longrightarrow x\!=\!y.
\end{align*}
\vspace*{-14.5mm}\\
\end{proof}
\begin{rema}
The preorder $\preceq_d$ is even a total or linear preorder on $\mM$. Its construction is based on the real function $f_d:\mM\to\mR_{\ge0}$. Therefore, we get $d(1,x)\le d(1,y)$ or $d(1,y)\le d(1,x)$ for all $x,y\in\mM$, and consequently $x \preceq_d y$ or $y \preceq_d x$.

We should also mention that the same relation $\preceq_d$ can be induced by different\lb generating functions. The preorders $\preceq_d$ and $\preceq_{\lambda\cdot d}$ are equivalent for $\lambda\in\mR>0$\lb because 
\vspace*{-3mm}
\begin{align*}
f_d(x)\le f_d(y) &\iff d(1,x)\le d(1,y)\\
&\iff \lambda\cdot d(1,x)\le\lambda\cdot d(1,y) \iff f_{\lambda\cdot d}(x)\le f_{\lambda\cdot d}(y).
\end{align*}
\vspace*{-10mm}
\end{rema}

So, specific arithmetic functions can also be used to generate preorders on $\mM$.\lb We will address this aspect in particular in Section~\refs{AFO}.\\\\

In the following technical lemma, we prove the equivalence of four implications that can be useful for proving further conclusions.

\begin{lemm} \label{equiv_PO}
Let $1,x,y\in\mM$, $d$ be a pseudometric on $\mM$, and $\preceq_d$ the corresponding induced preorder on $\mM$. Then, the following conditions are equivalent.
\vspace*{-3mm}
\begin{align*}
&(a)\quad d(x,y)\!=\!0 &&\Longrightarrow x\!=\!y \hspace{30mm}\text{, i.e.\quad $d$ is a metric},\\
&(b)\quad d(1,x)\!=\!d(1,y) &&\Longrightarrow x\!=\!y \hspace{30mm}\text{,}\\
&(b)\quad f_d(x)\!=\!f_d(y) &&\Longrightarrow x\!=\!y \hspace{30mm}\text{, and}\\
&(d)\quad x \preceq_d y \land y \preceq_d x &&\Longrightarrow x\!=\!y \hspace{30mm}\text{, i.e.\quad $\preceq_d$ is a partial order}.
\end{align*}
\vspace*{-14mm}\\
\end{lemm}
\begin{proof}
Given that $d$ is a pseudometric and $\preceq_d$ is a preorder, the reverse implications \enquote{$\Longleftarrow$} apply to the above conditions. The implied assertions are
\vspace*{-3mm}
\begin{align*}
&(a)\quad d(x,y)\!=\!0 &&\iff x\!=\!y,\\
&(b)\quad d(1,x)\!=\!d(1,y) &&\iff x\!=\!y,\\
&(c)\quad f_d(x)\!=\!f_d(y) &&\iff x\!=\!y,\text{ and}\\
&(d)\quad x \preceq_d y \land y \preceq_d x &&\iff x\!=\!y.
\end{align*}
\vspace*{-8mm}\\
(a $\iff$ b)

Given (a), $x=y$, and therefore $d(x,y)=0$. Then $d(1,x)=d(1,y)$ should also hold. We assume the contrary $d(1,x)\ne d(1,y)$ and set w.l.o.g. $d(1,x) < d(1,y)$. According to the triangle inequality, we conclude $d(1,x) < d(1,y) \le d(1,x)+d(x,y)$, i.e. $0<d(x,y)$ contrary to the assumption $d(x,y)=0$. 
\vspace*{1mm}\\
(b $\iff$ c)

We have  $f_d(x)\!=\!d(1,x)$ by definition.\\
Therefore, $d(1,x)\!=\!d(1,y) \iff f_d(x)\!=\!f_d(y)$.
\vspace*{1mm}\\
(b $\iff$ d)

It again suffices to prove the equivalence of the left-hand sides.
\vspace*{1mm}\\
$d(1,x)\!=\!d(1,y) \iff d(1,x)\le d(1,y) \land d(1,y)\le d(1,x) \iff x \preceq_d y \land y \preceq_d x.$
\vspace*{-5mm}\\
\end{proof}
\ \\

We now turn our attention to the divisibility properties with respect to a given\lb modulus $n\in\mN$. Pseudometrics on $\mD_n$, the set of divisors of $n$, play a central role here. In the following, we therefore restrict ourselves to those pseudometrics on $\mM$ that are extensions of pseudometrics on $\mD_n$. 

\begin{defi} \label{extended_function}
Let $n\in\mN$, $d$ be a pseudometric on $\mD_n$, and $x,y\in\mZ$. We define
\vspace*{-3mm}
\[d^{\around{n}}(x,y)=d\bigl(\gcd(x,n),\gcd(y,n)\bigr).\]
\end{defi}\vspace*{-1mm}

This definition of an extension includes the projections of $d^{\around{n}}$ onto any subset\lb $\mM\subseteq\mZ$ containing $\mD_n$. Furthermore, $d$ and $d^{\around{n}}$ coincide on $\mD_n$ by definition because $\gcd(x,n)=x\in\mD_n$ for all $x\in\mD_n$. Then $d^{\around{n}}(x,y)=d(x,y)$ for $x,y\in\mD_n$.

The extended function is also periodic by design.\vspace*{-3mm}
\[d^{\around{n}}(x+n,y)=d^{\around{n}}(x,y)=d^{\around{n}}(x,y+n).\]

Every pseudometric on $\mD_n$ can be extended to a pseudometric on $\mM$ if $\mD_n \subseteq \mM$ using the distance function from Definition~\refs{extended_function}. It reflects the focus on divisibility properties.

\begin{lemm} \label{extended_pseudometric}
Let $n\in\mN$, $d$ be a pseudometric on $\mD_n$, and $x,y\in\mM\supseteq\mD_n$. Then, the function $d^{\around{n}}$ extended from $d$ is a pseudometric on $\mM$.
\end{lemm}
\begin{proof}
$\gcd(x,n),\ \gcd(y,n)\in \mD_n$. Identity, symmetry, and triangle inequality are\lb satisfied because $d$ is a pseudometric on $\mD_n$.
\end{proof}
\begin{rema}
There can exist $x,y\in\mM$ such that $d^{\around{n}}(x,y)\!=\!0$ but $x\!\neq\! y$. If there is an $x\in\mM$ with $y=x+n\in\mM$, then we get
\vspace*{-3mm}
\[d^{\around{n}}(x,y)=d\bigl(\gcd(x,n),\gcd(x+n,n)\bigr)=d\bigl(\gcd(x,n),\gcd(x,n)\bigr)=0.\]\vspace*{-8mm}\\
In this case, $d^{\around{n}}$ cannot be a metric on $\mM$ even if $d$ were a metric on $\mD_n$.

The extension of a pseudometric $d$ on $\mD_n$ to a pseudometric $d^{\around{n}}$ on $\mM$ is defined\lb for every $\mM\subseteq\mZ$, so also for $\mZ$ itself.

We further emphasize that each extended pseudometric is uniquely defined by its values for all pairs $x,y\in\mD_n$. The circled superscript is intended to indicate the set of divisors which it is based on.

The trivial case $n=1$ always leads to the pseudometric $d^{\around{n}}(x,y)\equiv0$ for all $x,y\in\mZ$. Divisibility by $1$ is not really interesting either. We therefore exclude this case from our further considerations and generally require $n>1$ for the sake of simplicity.
\end{rema}\vspace*{6mm}

Analogously to the extension of a pseudometric according to Lemma~\refs{extended_pseudometric}, a given preorder $\preceq_d$ on $\mD_n$, induced by a pseudometric $d$ on $\mD_n$ as defined in Lemma~\refs{induced_preorder}, can be extended to a preorder $\preceq^*_d$ on $\mM$ for any $n\in\mN>1$.

\begin{lemm} \label{extended_relation}
Let $n\in\mN>1$, $d$ be a pseudometric on $\mD_n$, and $\preceq_d$ the preorder on $\mD_n$ induced by $d$. Furthermore, $x,y\in\mM\supseteq\mD_n$. Then, the relation $\preceq^*_d$ extended from $\preceq_d$
\vspace*{-3mm}
\[x\preceq^*_d y \iff \gcd(x,n)\preceq_d\gcd(y,n)\]\vspace*{-8mm}\\
is a preorder on $\mM$.
\end{lemm}
\begin{proof}
$\gcd(x,n),\ \gcd(y,n)\in \mD_n$. Reflexivity and transitivity of $\preceq^*_d$ are satisfied\lb because $\mD_n$ is pre-ordered by $\preceq_d$.
\end{proof}
\begin{rema}
There can exist $x,y\in\mM$ such that $x\preceq^*_d y$ and $y\preceq^*_d x$ but $x\neq y$. If there is an $x\in\mM$ with $y=x+n\in\mM$, then we get
\vspace*{-3mm}
\begin{align*}
x\preceq^*_d y &\iff \gcd(x,n)\preceq_d\gcd(x+n,n) \iff \gcd(x,n)\preceq_d\gcd(x,n)\\
&\iff \gcd(x+n,n)\preceq_d\gcd(x,n) \iff y\preceq^*_d x.
\end{align*}
\vspace*{-8mm}\\
In this case, $\preceq^*_{d_{(n)}}$ cannot be an antisymmetric relation on $\mM$ even if $\preceq_d$ were anti-\lb symmetric on $\mD_n$, i.e. if $\preceq_d$ were a partial order on $\mD_n$.
\end{rema}

For any $n\in\mN>1$, the preorder $\preceq_{d^{\around{n}}}$ on $\mM$ induced by the extension $d^{\around{n}}$ of a given pseudometric $d$ on $\mD_n$ turns out to be the extension $\preceq^*_d$ of the preorder $\preceq_d$ on $\mD_n$ induced by the pseudometric $d$ on $\mD_n$. The following diagram commutes.

\begin{figure}[H]
  \centering\vspace*{-3mm}
\begin{gather*}\xymatrix{
  \mD_n & \subseteq & \mM \\
  d \ar[0,2]^{\ref{extended_pseudometric}}
  \ar[2,0]_{\ref{induced_preorder}}
    & & d^{\around{n}} \ar[2,0]^{\ref{induced_preorder}} \\
  \\
  \preceq_d \ar[0,2]_{\ref{extended_relation}} & & \preceq_{d^{\around{n}}} 
}\end{gather*}
   \caption{Extension of pseudometrics and preorders.}\label{dia1}\ 
\end{figure}

\begin{prop}
Let $n\in\mN>1$, $d$ be a pseudometric on $\mD_n$, and $x,y\in\mM\supseteq\mD_n$. Then, the relations $\preceq^*_d$ according to Lemmata~\refs{induced_preorder} and \refs{extended_relation}, and $\preceq_{d^{\around{n}}}$ according to Lemmata~\refs{extended_pseudometric}\lb and \refs{induced_preorder} are identical.
\[x\preceq^*_d y \iff x\preceq_{d^{\around{n}}} y.\]
\end{prop}
\begin{proof}
According to Lemma~\refs{extended_relation}, we know $x\preceq^*_d y \iff \gcd(x,n)\preceq_d\gcd(y,n)$.\lb On the other hand, we get
\vspace*{-3mm}
\begin{align*}
x\preceq_{d^{\around{n}}} y & \iff d^{\around{n}}(1,x)\le d^{\around{n}}(1,y) & \text{by Lemma~\refs{induced_preorder},}\\
& \iff d\bigl(\gcd(1,n),\gcd(x,n)\bigr)\le d\bigl(\gcd(1,n),\gcd(y,n)\bigr) & \text{by Lemma~\refs{extended_pseudometric},}\\
& \iff d\bigl(1,\gcd(x,n)\bigr)\le d\bigl(1,\gcd(y,n)\bigr),\\
& \iff  \gcd(x,n)\preceq_d \gcd(y,n) & \text{by Lemma~\refs{induced_preorder}.}
\end{align*}
\vspace*{-14.5mm}\\
\end{proof}


\section{Equivalences and partitions\phj} \label{EP}

We recall basic notions about relations and corresponding partitions, see e.g. \cite{Schechter_1997, Burris_Sankappanavar_2012}.\lb A reflexive, symmetric, and transitive relation $\sim$ on $\mM$ with
\vspace*{-2mm}
\begin{itemize}\itemsep-6mm
\item $x \sim x$ \tab reflexivity,\\
\item  $x \sim y \Longrightarrow y \sim x$ \tab symmetry, \quad and\\
\item $x \sim y \land y \sim z \Longrightarrow x \sim z$ \tab transitivity
\end{itemize} \vspace*{-2mm}
is called \textbf{equivalence relation}.

The corresponding set of equivalence classes $\{ [x]_\sim=\{ y\in\mM : y\sim x \} : x\in\mM \}$ form a partition of S. This partition is called the \textbf{quotient set} $\mM\qs\sim_n$.\\

Let $\mM_1,\mM_2\subseteq\mZ$, and $x,y\in\mM_1$. Two equivalence relations $\sim_1$ on $\mM_1$ and $\sim_2$ on $\mM_2$ or their corresponding quotient sets $\mM_1\qs\sim_1$ and $\mM_2\qs\sim_2$ are called \textbf{isomorphic}, i.e.
\vspace*{-3mm}
\[\text{\qquad} \sim_1 \ \simeq\  \sim_2 \text{\qquad and\qquad} \mM_1\qs\sim_1 \ \simeq\  \mM_2\qs\sim_2\]
\vspace*{-7mm}\\
if there is a relation-preserving morphism $\rho:\mM_1\to\mM_2$ with\\
$x\sim_1 y \iff \rho(x)\sim_2\rho(y)$ or equivalently $[x]_{\sim_1} = [y]_{\sim_1} \iff [\rho(x)]_{\sim_2} = [\rho(y)]_{\sim_2}$.\\

Analogous to the refinement of partitions, we describe the refinement of equivalence relations and quotient sets as follows. Let $\sim_1$ and $\sim_2$ be equivalence relations on $\mM$ with their corresponding quotient sets $\mM\qs\sim_1$ and $\mM\qs\sim_2$, respectively. We call $\sim_1$ a \textbf{refinement} of $\sim_2$ if and only if $\mM\qs\sim_1$ is a refinement of $\mM\qs\sim_2$, i.e. every equivalence class of $\mM\qs\sim_1$ is included in an equivalence class of $\mM\qs\sim_2$. We write\\
\vspace*{-3mm}
\[\text{\qquad} \sim_1 \ \sqsubseteq\  \sim_2 \text{\qquad and\qquad} \mM\qs\sim_1 \ \sqsubseteq\  \mM\qs\sim_2\]
\vspace*{-5mm}\\
if $[x]_{\sim_1} \subseteq [x]_{\sim_2}$ holds for all $x\in\mM$.\\\\

For pseudometrics and preorders discussed in the previous section, we investigate specific equivalence relations and its related quotient sets. The general notation, we define as follows.

\begin{lemm} \label{ER_QS}
Let $1,x,y\in\mM$, and $d$ be a pseudometric and $\preceq_d$ the induced preorder on $\mM$. Then, the relations $\sim_d$ and $\sim_{\preceq_d}$ defined by
\vspace*{-2mm}
\begin{align*}
x \sim_d y &\iff d(x,y)=0 \text{\qquad and}\\
x \sim_{\preceq_d} y &\iff x\preceq_d y \land y\preceq_d x
\end{align*}
\vspace*{-7mm}\\
are equivalence relations.
\end{lemm}
\begin{proof}\ \\
The relation $\sim_d$ is

reflexive: $x \sim_d x$ because $d(x,x)=0$,

symmetric: $x \sim_d y \iff d(x,y)=0 \iff d(y,x)=0 \iff y \sim_d x$, \quad and

transitive: $x \sim_d y \land y \sim_d z \iff d(x,y)=d(y,z)=0\\
\text{\hspace{22mm}} \iff 0=d(x,y)+d(y,z)\ge d(x,z) \iff d(x,z)=0 \iff x \sim_d z.$\newpage
\noindent
The relation $\sim_{\preceq_d}$ is

reflexive: $x \sim_{\preceq_d} x$ because $x\preceq_d x$,

symmetric: $x \sim_{\preceq_d} y \iff x\preceq_d y \land y\preceq_d x \iff y\preceq_d x \land x\preceq_d y\\
\text{\hspace{22mm}} \iff y \sim_{\preceq_d} x$, \quad and

transitive: $x \sim_{\preceq_d} y \land y \sim_{\preceq_d} z \iff x\preceq_d y \land y\preceq_d x \ \land\  y\preceq_d z \land z\preceq_d y\\
\text{\hspace{58mm}} \iff x\preceq_d y \land y\preceq_d z \ \land\  z\preceq_d y \land y\preceq_d x\\
\text{\hspace{22mm}} \iff x\preceq_d z \land z\preceq_d x \iff x \sim_{\preceq_d} z.$
\end{proof}\vspace*{1mm}

Given a pseudometric $d$ on $\mM$, an equivalent definition of $\sim_{\preceq_d}$ can be directly\lb derived from the resulting centred pseudometric or its generating function $f_d$. The definition of $\sim_{\preceq_d}$ in Lemma~\ref{induced_preorder} uses only the referred generating function.
\begin{lemm} \label{crit_QO}
Let $1,x,y\in\mM$, and $d$ be a pseudometric and $\preceq_d$ the induced preorder on $\mM$. Then,
\[x \sim_{\preceq_d} y \iff d(1,x)=d(1,y) \iff f_d(x)=f_d(y).\]
\end{lemm}
\begin{proof}\ 
\vspace*{-9mm}
\begin{align*}
x \sim_{\preceq_d} y &\iff x\preceq_d y \land y\preceq_d x\\
& \iff f_d(x)\le f_d(y) \land f_d(y)\le f_d(x)\\
& \iff f_d(x)=f_d(y) \iff d(1,x)=d(1,y).
\end{align*}
\vspace*{-14mm}\\
\end{proof}
\vspace*{8mm}

Pseudometrics on $\mD_n$ can be extended to pseudometrics on $\mM\supseteq\mD_n$ using $\gcd$\lb according to Definition~\refs{extended_pseudometric}. The induced preorders are also connected by $\gcd$\lb according to Definition~\refs{extended_relation}. Therefore, we can proof an isomorphism between related quotient sets as implied by Figure~\refs{dia1}.

We refer to the general notation of extended pseudometrics according to\lb Definition~\refs{extended_function}.\\[-3mm]

\begin{prop} \label{extension_isomorphism}
Let $n\in\mN>1$, $d^{\around{n}}$ be a pseudometric on $\mZ$ extended from a pseudometric $d$ on $\mD_n$, and $\mM\supseteq\mD_n$. Then,
\vspace*{-3mm}
\begin{align*}
(a)\text{\hspace{22.5mm}} \mM\qs\sim_{d^{\around{n}}} &\ \simeq\  \mD_n\qs\sim_d \text{,\qquad and}\\
(b)\text{\hspace{20mm}} \mM\qs\sim_{\preceq_{d^{\around{n}}}} &\ \simeq\  \mD_n\qs\sim_{\preceq_d}.
\end{align*}
\vspace*{-7mm}\\
Equivalently, \quad$\sim_{d^{\around{n}}}\ \simeq\  \sim_d$\quad and \quad$\sim_{\preceq_{d^{\around{n}}}}\ \simeq\  \sim_{\preceq_d}.$
\end{prop}

\begin{proof} Let $x,y\in\mM$. Then $x\to\gcd(x,n)\in\mD_n$ is a relation-preserving morphism.\\\\
(a)
\vspace*{-3mm}
\begin{align*}
[x]_{\sim_{d^{\around{n}}}} = [y]_{\sim_{d^{\around{n}}}} &\iff x \sim_{d^{\around{n}}} y \iff d^{\around{n}}(x,y)=0 \\
&\iff d(\gcd(x,n),\gcd(y,n))=0 \iff \gcd(x,n) \sim_d \gcd(y,n)\\
&\iff [\gcd(x,n)]_{\sim_d} = [\gcd(y,n)]_{\sim_d}.
\end{align*}
\vspace*{-8mm}\\
(b)
\vspace*{-3mm}
\begin{align*}
[x]_{\sim_{\preceq_{d^{\around{n}}}}} = [y]_{\sim_{\preceq_{d^{\around{n}}}}} &\iff x \sim_{\preceq_{d^{\around{n}}}} y \iff x\preceq_{d^{\around{n}}} y \land y\preceq_{d^{\around{n}}} x\\
&\iff \gcd(x,n)\preceq_d\gcd(y,n)) \land \gcd(y,n)\preceq_d\gcd(x,n)) \hspace{17mm}\\
&\iff \gcd(x,n) \sim_{\preceq_d} \gcd(y,n)\\
&\iff [\gcd(x,n)]_{\sim_{\preceq_d}} = [\gcd(y,n)]_{\sim_{\preceq_d}}.
\end{align*}
\vspace*{-14mm}\\
\end{proof}
\vspace*{1mm}

Another conclusion can be derived from Figure~\ref{dia1} with respect to the definition\lb in Lemma~\refs{induced_preorder}. Any equivalence relation based on a pseudometric according to\lb Lemma~\ref{ER_QS} turns out to be a refinement of the equivalence relation based on the\lb induced preorder.

\begin{prop} \label{refinement}
Let $d$ be a pseudometric on $\mM$, and $\preceq_d$ the corresponding induced preorder. Then,
\[\mM\qs\sim_d \ \sqsubseteq\  \mM\qs\sim_{\preceq_d}.\]
\end{prop}
\begin{proof}
Let $1,x,y\in\mM$. According to the triangle inequality, we have\vspace*{1mm}\\
\hspace*{10mm}$d(1,x) \le d(1,y)+d(y,x)$\quad and \quad$d(1,x)+d(x,y) \ge d(1,y)$.\vspace*{1mm}\\
Then $[x]_{\sim_d} \subseteq [x]_{\sim_{\preceq_d}}$ because
\vspace*{-3mm}
\begin{align*}
y\in[x]_{\sim_d} &\iff x \sim_d y \iff d(x,y)=0\\
&\ \,\Longrightarrow\ \, d(1,x) \le d(1,y)+0\ \land\ d(1,y) \le d(1,x)+0\\
&\iff d(1,x)\le d(1,y)\ \land\ d(1,y)\le d(1,x) \iff d(1,x)=d(1,y)\\
&\underset{\ref{crit_QO}}{\iff} x \sim_{\preceq_d} y \iff y\in[x]_{\sim_{\preceq_d}}.
\end{align*}
\vspace*{-12mm}\\
\end{proof}

\begin{rema}
Let $n\in\mN>1$. Applying Proposition~\refs{extension_isomorphism}, we get
\vspace*{-3mm}
\begin{align*}
\mD_n\qs\sim_d &\ \sqsubseteq\  \mD_n\qs\sim_{\preceq_d} \text{,\qquad and}\\
\mM\qs\sim_{d^{\around{n}}} &\ \sqsubseteq\  \mM\qs\sim_{\preceq_{d^{\around{n}}}}.
\end{align*}
\vspace*{-8mm}\\
Analoguosly hold \quad$\sim_d \ \sqsubseteq\  \sim_{\preceq_d}$\quad and \quad$\sim_{d^{\around{n}}} \ \sqsubseteq\  \sim_{\preceq_{d^{\around{n}}}}.$
\end{rema}
\ \\[-6mm]

We summarise the latest results similar to Figure~\refs{dia1}.
\vspace*{3mm}

\begin{figure}[H]
  \centering\vspace*{-3mm}
\begin{gather*}\xymatrix{
  \mD_n & \subseteq & \mM \\
  \sim_d \ar[0,2]^{\text{$\ \simeq\ $}}_{\ref{extension_isomorphism}}
  \ar[2,0]^{\ref{refinement}}_{\rotatebox{270}{$\ \sqsubseteq\ $}}
    & & \sim_{d^{\around{n}}} \ar[2,0]^{\rotatebox{270}{$\ \sqsubseteq\ $}}_{\ref{refinement}} \\
  \\
 \sim_{\preceq_d} \ar[0,2]^{\ref{extension_isomorphism}}_{\text{$\ \simeq\ $}}
   & & \sim_{\preceq_{d^{\around{n}}}}
}\end{gather*}
   \caption{Extensions and refinements of equivalence relations.}\label{dia2}\ 
\end{figure}
\newpage

Below we discuss quotient sets and their refinements in the context of the\lb equivalence relations defined above. The respective partitions of the underlying\lb set $\mM$ essentially depend on the existence of pairs \ $x,y\in\mM$ \ with \ $d(x,y)=0$ \ or $d(1,x)=d(1,y)$.

If $d$ is even a metric on $\mM$ then there is an isomorphism between $\mM$ and the quotient set $\mM\qs\sim_d$. Each element of $\mM$ builds its own equivalence class. The quotient set $\mM\qs\sim_d$ represents the finest possible partition of $\mM$.

\begin{lemm} \label{metric_isomorphism}
Let $d$ be a metric on $\mM$. Then,
\vspace*{-3mm}
\[\mM \ \simeq\ \mM\qs\sim_d.\]
\vspace*{-8mm}
\end{lemm}
\begin{proof}
Let $x,y\in\mM$. In a metric, we have \ $d(x,y)=0 \!\iff\! x=y$ \ by definition.\lb Then also $x \sim_d y \iff d(x,y)=0 \iff x=y$.
\vspace*{-5mm}\\
\end{proof}
\begin{rema}
If furthermore $1\in\mM$ and $d(1,x)\!=\!d(1,y) \Longrightarrow x\!=\!y$, or equivalently\lb $f_d(x)\!=\!f_d(y) \Longrightarrow x\!=\!y$, then $\preceq_d$ is even a partial order by Lemma~\refs{induced_preorder}.\vspace*{1mm}

In this case, we also get $d(x,y)=0 \Longrightarrow x=y$ because of Lemma~\refs{equiv_PO}. Then
\vspace*{-3mm}
\[\mM \ \simeq\ \mM\qs\sim_d \ \ \simeq\ \mM\qs\sim_{\preceq_d}.\]
\end{rema}

In general, $\sim_d$ is a refinement of $\sim_{\preceq_d}$ according to Proposition~\refs{refinement}. The quotient set $\mM\qs\sim_{\preceq_d}$ aggregates integers that cannot be distinguished by the respective preorder $\sim_{\preceq_d}$. It is invariant against permutation of the function values of $f_d$ on $\mM$. The relation $\sim_{\preceq_d}$ is completely determined by them.

\begin{lemm} \label{basic_permutation}
Let $1,x\in\mM$, and $d_1,d_2$ be pseudometrics on $\mM$ with $f_{d_2}(x)=\nu\bigl(f_{d_1}(x)\bigr)$\lb where $\nu$ is a permutation of $f_{d_1}(\mM)$. Then
\vspace*{-3mm}
\[\mM\qs\sim_{\preceq_{d_1}}\ = \ \mM\qs\sim_{\preceq_{d_2}}.\]
\vspace*{-8mm}
\end{lemm}
\begin{proof}
As a consequence of Lemma~\refs{crit_QO}, we get
\vspace*{-3mm}
\begin{align*}
x \sim_{\preceq_{d_1}} y &\iff f_{d_1}(x)=f_{d_1}(y) \iff \nu\bigl(f_{d_1}(x)\bigr)=\nu\bigl(f_{d_1}(y)\bigr)\\
&\iff f_{d_2}(x)=f_{d_2}(y) \iff x\preceq_{d_2} y.
\end{align*}
\vspace*{-14mm}\\
\end{proof}

The coarsest conceivable partition of $\mM$ is the trivial partition $\{\mM\}$, in which all elements of $\mM$ fall into the same, single equivalence class. It can only be generated\lb by the trivial pseudometric.

\begin{lemm} \label{coarsest}
Let $d$ be a pseudometric on $\mM$ and $1\in\mM$. Then
\vspace*{-3mm}
\[\{\mM\} \ =\ \mM\qs\sim_d \ \ =\ \mM\qs\sim_{\preceq_d}
\qquad\iff\qquad
\forall \ x,y\in\mM\ :\ d(x,y)=0.\]
\vspace*{-8mm}
\end{lemm}
\begin{proof}
For all $x,y\in\mM$, we get
\vspace*{-3mm}
\begin{align*}
d(x,y)=0 &\iff d(1,x)=d(1,y)=0 \iff f_d(x)=f_d(y)=0\\
&\iff x\sim_d y \iff x\sim_{\preceq_d} y.
\end{align*}
\vspace*{-14mm}\\
\end{proof}


\section{Arithmetic functions and pseudometrics\phj} \label{AFM}

Arithmetic or number-theoretic functions are widely studied, e.g. \cite{Hildebrand_2013, Hardy_Wright_1975, McCarthy_1986, Schwarz_Spilker_1994, Cira_Smarandache_2016}. They are generally considered as complex-valued functions defined on the set of\lb natural numbers. The set of arithmetic functions together with pointwise\lb addition and convolution forms a commutative ring with unity. Considering pointwise addition, pointwise multiplication, and scalar multiplication results in a $\mC$-algebra instead.

Our focus will be on examining specific additive and multiplicative subsemigroups of this $\mC$-algebra and relating them to questions of metrics as discussed in Section~\refs{MO}. We will explore other topics such as orders, equivalences, and partitions in the next section. In this paper, we generally restrict the definitions to real-valued arithmetic functions $f:\mN\to\mR$.\\

Let $f,g:\mN\to\mR$, and $x,y,z\in\mN$. The (pointwise) \textbf{sum} and \textbf{product} of arithmetic functions are defined as $\quad(f+g)(x)=f(x)+g(x) \text{\quad and \quad} (f\cdot g)(x)=f(x)\cdot g(x)$ while their \textbf{convolution} is
\vspace*{-1mm}
\[ (f*g)(x)=\sum_{z\mid x} f(z)\mdot g(\frac{x}{z}).\]
\vspace*{-8mm}\\

As usual, additive and multiplicative arithmetic functions are characterised as\lb follows.
\vspace*{-2mm}
\begin{itemize}\itemsep-6mm
\item \textbf{additive} \tab $x\copr y\ \ \Longrightarrow\  f(x\mdot y)=f(x)\mplus f(y)$,\\
\item \textbf{strongly additive} \tab $\forall \ x,y\ \;\; : \,\quad f(x\mdot y)=f(x)\mplus f(y)$,\\
\item \textbf{multiplicative} \tab $x\copr y\ \ \Longrightarrow\  g(x\mdot y)=g(x)\ \mdot\  g(y)$,\\
\item \textbf{strongly multiplicative} \tab $\forall \ x,y\ \;\; : \,\quad g(x\mdot y)=g(x)\ \mdot\  g(y)$.
\end{itemize} \vspace*{-1mm}
Furthermore, the additive zero function $\mathit{0}(x)=0$ for all $x\in\mN$ is generally excluded from the definition of multiplicative arithmetic functions \cite{Hildebrand_2013, McCarthy_1986, Schwarz_Spilker_1994}. Then for all\lb multiplicative functions $g$, there exists an $y\in\mN$ with $g(y)\ne0$ and
\vspace*{-3mm}
\[g(y)=g(1\mdot y)=g(1)\mdot g(y).\]
\vspace*{-8mm}\\
So, $g(1)=1$ for all multiplicative arithmetic functions. As a conclusion of
\vspace*{-3mm}
\[f(1)=f(1\mdot 1)=f(1)\mplus f(1),\]
\vspace*{-8mm}\\
we also know $f(1)=0$ for all additive arithmetic functions $f$. Thus, additive and multiplicative functions form disjoint sets.

Moreover, additivity and multiplicativity are preserved under pointwise sum or product, respectively, because
\vspace*{-2mm}
\begin{align*}
(f_1+f_2)(x\mdot y)&=f_1(x\mdot y)+f_2(y\mdot x)=f_1(x)+f_1(y)+f_2(x)+f_2(y)\\
&=(f_1+f_2)(x)+(f_1+f_2)(y),\\
(g_1\cdot g_2)(x\mdot y)&=g_1(x\mdot y)\cdot g_2(y\mdot x)=g_1(x)\cdot g_1(y)\cdot g_2(x)\cdot g_2(y)\\
&=(g_1\cdot g_2)(x)\cdot (g_1\cdot g_2)(y).
\end{align*} \vspace*{-3mm}\\

In Section~\refs{MO}, we described the use of functions to generate pseudometrics and\lb preorders on sets of integers. We now apply these findings to certain arithmetic\lb functions.

Because of Lemma~\refs{function_metric}, arithmetic functions $f$ with $f(1)=0$ and $f(x)\ge0$ for $x>1$ generate pseudometrics or even metrics
\vspace*{-3mm}
\begin{align*}
d(x,y)=\begin{cases}
   0 &\text{for $x=y$, \quad and}\\
   f(x)+f(y) &\text{for $x\ne y$} \end{cases}
\end{align*} \vspace*{-3mm}\\
on integer subsets $\mM$ with  $\{1\}\subseteq\mM\subseteq\mZ$, so also on any $\mD_n$. Such a pseudometric\lb or metric on $\mD_n$ can be extended to a pseudometric $d^{\around{n}}$on $\mZ$ by
\vspace*{-3mm}
\begin{align*}
d^{\around{n}}(x,y)=d\bigl(\gcd(x,n),\gcd(y,n)\bigr)=\begin{cases}
   0 &\text{for $x=y$, \quad and}\\
   f\bigl(\gcd(x,n)\bigr)+f\bigl(\gcd(y,n)\bigr) &\text{for $x\ne y$} \end{cases}
\end{align*} \vspace*{-3mm}\\
according to Lemma~\refs{extended_pseudometric}. We summarise the results in the following corollary.\\

\begin{defi} \label{non-negative}
Let $x\in\mN>1$. We signify the set of arithmetic functions $f:\mN\to\mR$ with $f(1)=0$ and $f(x)\ge0$ as
\vspace*{-2mm}
\[\msG=\{f:f(x)\ge0\land f(1)=0\}.\]
\vspace*{-6mm}\\
Furthermore,
\vspace*{-5mm}
\[\msG_0=\{f\in\msG:f(x)>0\}.\]
\vspace*{-3mm}
\end{defi}

\begin{coro} \label{coro_metric1}
Let $n\in\mN>1$,  $x,y\in\mZ$, and $f\in\msG$. Then
\vspace*{-3mm}
\[d^{\around{n}}(x,y)=\begin{cases}
   0 &\text{for $x=y$, \quad and}\\
   f\bigl(\gcd(x,n)\bigr)+f\bigl(\gcd(y,n)\bigr) &\text{for $x\ne y$} \end{cases}
\]
\vspace*{-4mm}\\
is a pseudometric on $\mZ$, so also on $\mD_n$.

If $f\in\msG_0$ then $d^{\around{n}}$ is a metric on $\mD_n$.
\end{coro}
\begin{proof}
The assertions follow from Lemmata~\ref{function_metric}~and~\refs{extended_pseudometric}.
\end{proof}
\ \\

The requirement \ $f(1)=0$ \ is satisfied by all additive arithmetic functions $f$.\lb So, non-negative additive functions form a subset of $\msG$.

For multiplicative arithmetic functions $g$ or $h$ instead, we have $g(1)=h(1)=1$. Two simple transformations, $f=g-1$ and $f=1-h$, yield $f(1)=0$. So, multiplicative functions can also be used to generate pseudometrics on $\mD_n$ if $f(x)\ge0$ is guaranteed after transformation for $x>1$. In the first case, we need $g(x)\ge1$ for $x>1$. The latter case results in $h(x)\le1$ at first. But if e.g. $h(x)<0$ and $h(y)<1/h(x)$ for some coprime $x,y>1$ then $h(x\mdot y)=h(x)\cdot h(y) >1$. So here, we need to require $0\le h(x)\le1$ for $x>1$.

We will show that both sets of multiplicative functions are in some sense equivalent to the non-negative additive functions and want to further explore all three in terms of metrics.

\begin{defi} \label{admissible}
Let $x\in\mN>1$. We define the following subsets of arithmetic functions and call its elements admissible.
\vspace*{-1mm}
\begin{align*}
\msA\,\ \ &\;:\ \!\text{ admissible additive arithmetic functions.}\\
\msA\,\ \ &=\{ f:\text{$f$ is an additive arithmetic function and }f(x)\ge0\}.\\
\msA_0\ &=\{f\in\msA:f(x)>0\}.\\[1.3ex]
\msM\ \ &\;:\ \!\text{ admissible multiplicative arithmetic functions, bounded below.}\\
\msM\ \ &=\{ g:\text{$g$ is a multiplicative arithmetic function and }g(x)\ge1\}.\\
\msM_1\,&=\{g\in\msM:g(x)>1\}.\\[1.3ex]
\msI\ \ \ &\;:\ \!\text{ admissible multiplicative arithmetic functions, bounded by a finite interval.}\\
\msI\ \ \ &=\{ h:\text{$h$ is a multiplicative arithmetic function and }0\le h(x)\le1\}.\\
\msI_0\ \,&=\{h\in\msI:0<h(x)\le1\}.\\
\msI_1\ \,&=\{h\in\msI:0\le h(x)<1\}.\\
\msI_{01}&=\{h\in\msI:0<h(x)<1\}.
\end{align*} \vspace*{-5mm}\\
\end{defi}

By this definition, we have $\msA \subset \msG$ and $\msA_0 \subset \msG_0$. Moreover, $\msA_0 \subset \msA$, $\msM_1 \subset \msM$, $\msI_{01} \subset \msI_0 \subset \msI$, and $\msI_{01} \subset \msI_1 \subset \msI$. Since additive and multiplicative arithmetic functions are generally disjoint,  so are $\msG$ and $\msM\cup\msI$. Furthermore, $\msM\cap\msI=\msM\cap\msI_0=\{\mathit{1}\}$.\\

All sets defined in Definition~\ref{admissible} are closed under pointwise addition or\lb multiplication, respectively. The set $\msG$ according to Definition~\ref{non-negative} is closed both\lb under pointwise addition and under pointwise multiplication. Furthermore, certain sets of admissible arithmetic functions are isomorphic with respect to both, the binary operations and their additive or multiplicative properties.

\begin{lemm}
$(\msG,+)$, $(\msG,\,\cdot\,)$, $(\msG_0,\,\cdot\,)$, $(\msA,+)$, $(\msM,\,\cdot\,)$, $(\msI,\,\cdot\,)$, and $(\msI_0,\,\cdot\,)$ are commutative monoids.

$(\msG_0,+)$, $(\msA_0,+)$, $(\msM_1,\,\cdot\,)$, and $(\msI_{01},\,\cdot\,)$ are subsemigroups of them, respectively.
\end{lemm}
\begin{proof}
The binary operations are constructed pointwise from commutative and\lb associative operations on integers. Additivity and multiplicativity are preserved. The main point that remains to be proved is that the sets are closed under pointwise\lb operations and include an identity element.

Let $x,y,z\in\mN$ and $x>1$.\\[1.3ex]
$(\msG,+)$ : Let $f_1,f_2\in\msG$.

Then $(f_1+f_2)(1)=f_1(1)+f_2(1)=0$ \quad and \quad $(f_1+f_2)(x)=f_1(x)+f_2(x)\ge0$.

The identity element is $\mathit{0}(y)=0$. $(\msG_0,+)$ is not a monoid because $\mathit{0}\not\in\msG_0$.\\[0.5ex]
$(\msG,\,\cdot\,)$ : Let $f_1,f_2\in\msG$.

Then $(f_1\,\cdot\,f_2)(1)=f_1(1)\,\cdot\,f_2(1)=0$ \quad and \quad $(f_1\,\cdot\,f_2)(x)=f_1(x)\,\cdot\,f_2(x)\ge0$.

The identity element is $f_0$ with $f_0(0)=0$ and $f_0(x)=1$ for all $x\in\mN>1$.

Thus, $(\msG_0,\,\cdot\,)$ is also a monoid.\\[0.5ex]
$(\msA,+)$ : Let $f_1,f_2\in\msA$.

Then $(f_1+f_2)(1)=f_1(1)+f_2(1)=0$ \quad and \quad $(f_1+f_2)(x)=f_1(x)+f_2(x)\ge0$.

The identity element is $\mathit{0}(y)=0$. $(\msA_0,+)$ is not a monoid because $\mathit{0}\not\in\msA_0$.\\[0.5ex]
$(\msM,\,\cdot\,)$ : Let $g_1,g_2\in\msM$.

Then $(g_1\,\cdot\,g_2)(1)=g_1(1)\cdot g_2(1)=1$ \quad and $(g_1\,\cdot\,g_2)(x)=g_1(x)\cdot g_2(x)\ge1$.

The identity element $\mathit{1}(y)=1$. $(\msM_1,\,\cdot\,)$ is not a monoid because $\mathit{1}\not\in\msM_1$.\\[0.5ex]
$(\msI,\,\cdot\,)$ : Let $h_1,h_2\in\msI$.

Then $(h_1\,\cdot\,h_2)(1)=h_1(1)\cdot h_2(1)=1$ \quad and $0\le (h_1\,\cdot\,h_2)(x)=h_1(x)\cdot h_2(x)\le1$.

The identity element $\mathit{1}(y)=1$. $(\msI_{01},\,\cdot\,)$ is not a monoid because $\mathit{1}\not\in\msI_{01}$.\\[0.5ex]
$(\msI_0,\,\cdot\,)$ : Analogous to $(\msI,\,\cdot\,)$.\\[0.5ex]
$(\msI_1,\,\cdot\,)$ : Analogous to $(\msI_{01},\,\cdot\,)$.\\[-1.5ex]

The closure property of the semigroups follows in the same way, but the identity elements are not included.\vspace*{-5mm}\\
\end{proof}

\vspace*{1mm}
\begin{prop} \label{iso-binary}
$(\msA,+)$, $(\msM,\,\cdot\,)$, and $(\msI_0,\,\cdot\,)$ are isomorphic with respect to both, the binary operations and the additive and multiplicative properties of their elements. So are $(\msA_0,+)$, $(\msM_1,\,\cdot\,)$ and $(\msI_{01},\,\cdot\,)$
\end{prop}
\begin{proof}
Let $f,f_1,f_2\in\msA$, $g,g_1,g_2\in\msM$, $h,h_1,h_2\in\msI_0$, and $x,y\in\mN$.\\
The bijective morphisms

 $f\to g=\e^{\ f}\in\msM$ with the inverse  $g\to f=\ln(g)\in\msA$,

 $f\to h=\e^{-f}\in\msI_0$ with the inverse  $h\to f=-\ln(h)\in\msA$, and
 
 $g\to h=1/g\in\msI_0$ with the inverse  $h\to g=1/h\in\msM$\\
meet the requirements. The domains of definition of the functions are observed.

The isomorphisms preserve the binary operations of the semigroups.\\[0.5ex]
$\e^{(f_1+f_2)(x)}=\e^{\ f_1(x)+f_2(x)}=\e^{\ f_1(x)}\cdot \e^{\ f_2(x)}=g_1(x)\cdot g_2(x)=(g_1\cdot g_2)(x)$,\\[0.5ex]
$\e^{-(f_1+f_2)(x)}=\e^{-(f_1(x)+f_2(x))}=\e^{-f_1(x)}\cdot \e^{-f_2(x)}=h_1(x)\cdot h_2(x)=(h_1\cdot h_2)(x)$,\\[0.5ex]
$1/(g_1\mdot g_2)(x)\!=\!1/\bigl(g_1(x)\mdot g_2(x)\bigr)\!=\!\bigl(1/(g_1(x)\bigr)\mdot\bigl(1/g_2(x)\bigr)\!=\!h_1(x)\cdot h_2(x)\!=\!(h_1\mdot h_2)(x)$.\vspace{-4mm}\\

The isomorphisms also preserve the additive and multiplicative functional\lb equations of arithmetic functions. Let $x\perp y$.\\[0.5ex]
$\e^{\ f(x\cdot y)}=\e^{\ (f(x)+f(y))}=\e^{\ f(x)}\cdot \e^{\ f(y)}=g(x)\cdot g(y)=g(x\mdot y)$,\\[0.5ex]
$\e^{-f(x\cdot y)}=\e^{-(f(x)+f(y))}=\e^{-f(x)}\cdot \e^{-f(y)}=h(x)\cdot h(y)=h(x\mdot y)$,\\
$1/g(x\mdot y)=1/\bigl(g(x)\cdot g(y)\bigr)=\bigl(1/(g(x))\cdot(1/g(y)\bigr)=h(x)\cdot h(y)=h(x\mdot y)$.\vspace{-4mm}\\

The assertions also apply to subgroups, as these are closed.
\vspace*{-5mm}\\
\end{proof}
\ \\

In Corollary~\refs{coro_metric1}, we have summarised the main results of Section~\ref{MO} regarding the pseudometrics on $\mD_n$ induced by arithmetic functions. In the following, we focus on the generation of pseudometrics or metrics on $\mD_n$ for a given $n\in\mN$ using admissible arithmetic functions. This can be done in different ways.

In all cases, the known additive combination of function values according to\lb Corollary~\ref{coro_metric1} can be applied. In the case of admissible interval-bounded multiplicative functions, however, a multiplicative combination can also lead to pseudometrics. To point this out, we introduce a special notation for metrics derived from admissible arithmetic functions.

\begin{theo} \label{coro_metric2}
Let $n\in\mN>1$,  $x,y\in\mZ$, $f\in\msA$, $g\in\msM$, and $h\in\msI$. Then
\vspace*{-3mm}
\begin{align*}
d^{\around{n}}_{(f,+)}(x,y)&=\begin{cases}
   0 &\hspace*{7mm}\text{for $x=y$,}\\
   f\bigl(\gcd(x,n)\bigr)+f\bigl(\gcd(y,n)\bigr) &\hspace*{7mm}\text{for $x\ne y$}, \end{cases}\\
d^{\around{n}}_{(g,+)}(x,y)&=\begin{cases}
   0 &\text{for $x=y$,}\\
   g\bigl(\gcd(x,n)\bigr)+g\bigl(\gcd(y,n)\bigr)-2 &\text{for $x\ne y$, \quad and} \end{cases}\\
d^{\around{n}}_{(h,+)}(x,y)&=\begin{cases}
   0 &\text{for $x=y$,}\\
   2-h\bigl(\gcd(x,n)\bigr)-h\bigl(\gcd(y,n)\bigr) &\text{for $x\ne y$} \end{cases}
\end{align*} \vspace*{-4mm}\\
are pseudometrics on $\mZ$, so also on $\mD_n$.

If $f\in\msA_0$, $g\in\msM_1$, and $h\in\msI_1$ then they are metrics on $\mD_n$.
\end{theo}
\begin{proof}
With the transformations $f=g-1$ or $f=1-h$, respectively, the statements are included in those of Corollary~\refs{coro_metric1}. Metric properties are inherited by subsets.
\end{proof}

A special multiplicative combination of the values of functions $h\in\msI$ turns out\lb to be another way of generating pseudometrics on $\mZ$. This approach does not work with functions of $\msA$ or $\msM$.

\begin{theo} \label{coro_metric3}
Let $n\in\mN>1$,  $x,y\in\mZ$, and $h\in\msI$. Then
\vspace*{-3mm}
\begin{align*}
d^{\around{n}}_{(h,\,\cdot\,)}(x,y)&=\begin{cases}
   0 &\text{for $x=y$,}\\
   1-h\bigl(\gcd(x,n)\bigr)\cdot h\bigl(\gcd(y,n)\bigr) &\text{for $x\ne y$} \end{cases}
\end{align*} \vspace*{-4mm}\\
is a pseudometric on $\mZ$, so also on $\mD_n$.

If $h\in\msI_1$ then $d^{\around{n}}_{(h,\,\cdot\,)}$ is a metric on $\mD_n$.
\end{theo}
\begin{proof}
Identity and symmetry follow by construction. The triangle inequality remains to prove. W.l.o.g. let $x,y,z\in\mD_n$, i.e. $x,y,z\le n$.
\vspace*{-3mm}
\begin{align*}
d^{\around{n}}_{(h,\,\cdot\,)}(x,y)+d^{\around{n}}_{(h,\,\cdot\,)}(y,z)&-d^{\around{n}}_{(h,\,\cdot\,)}(x,z)=\\
&=1-h(x)\cdot h(y)+1-h(y)\cdot h(z)-\bigl(1-h(x)\cdot h(z)\bigr)\\
&=1+h(x)\cdot h(z)-h(x)\cdot h(y)-h(y)\cdot h(z)\\
&=1+h(x)\cdot h(z)-h(y)\cdot\bigl(h(x)+h(z)\bigr)\\
&\ge1+h(x)\cdot h(z)-\bigl(h(x)+h(z)\bigr)\\
&=1+h(x)\cdot\bigl(h(z)-1\bigr)-h(z)\\
&=\bigl(h(x)-1\bigr)\cdot\bigl(h(z)-1\bigr)\ge0
\end{align*} \vspace*{-8mm}\\
If $h\in\msI_1$ and $x\ne y$ then
$d^{\around{n}}_{(h,\,\cdot\,)}(x,y)=0\iff1=h(x)\cdot h(y)\iff h(x)=h(y)=1$, 
in contradiction to the domain of $\msI_1$.
\end{proof}

We want to emphasise that there exist pseudometrics on $\mD_n$ that cannot be\lb generated by admissible arithmetic functions. The respective generating functions form a strict subset of $\msG$. We give an example for $n=6$ and define $f(x)=x-1$. Then $f\in\msG$ and 
\[d(x,y)=\begin{cases}
   0 &\text{for $x=y$, \quad and}\\
   f(x)+f(y) &\text{for $x\ne y$} \end{cases}
\]\vspace*{-1mm}\\
is a pseudometric on $\in\mD_n$ according to Corollary~\refs{coro_metric1}. However, $f$ is the unique generating function because of Lemma~\refs{centred_pseudometric}, but it is neither multiplicative because $f(1)=0$ nor additive because $f(6)=5\ \ne\ f(2)+(3)=1+2=4$.

In Definition~\refs{admissible}, we introduced sets of admissible arithmetic functions. These sets are disjoint apart from the multiplicative identity function $\mathit{1}$. So, $\msA$, $\msM_1$, and $\msI_1$\lb are pairwise disjoint. Given $n\in\mN$, admissible arithmetic functions were used to generate four sets of pseudometrics on $\mZ$ according to Theorems~\ref{coro_metric2}~and~\refs{coro_metric3}. The corresponding sets of pseudometrics also do not overlap if $n$ has at least two prime divisors. This applies in particular to all included metrics because $\msA_0\subset\msA$.\\[-3mm]

\begin{prop} \label{disjoint_2}
Let $n\in\mN$ have at least two prime divisors. Then, the following four sets\lb of pseudometrics are pairwise disjoint on $\mD_n$.
\vspace*{-2mm}
\begin{align*}
\msD^{\!\around{n}}_{\!(\msA,+)}&=\bigl\{ d^{\around{n}}_{(f,+)} : f\in\msA \bigr\}, \\
\msD^{\!\around{n}}_{\!(\msM_1,+)}&=\bigl\{ d^{\around{n}}_{(g,+)} : g\in\msM_1 \bigr\}, \\
\msD^{\!\around{n}}_{\!(\msI_1,+)}&=\bigl\{ d^{\around{n}}_{(h,+)} : h\in\msI_1 \bigr\}, \\
\msD^{\!\around{n}}_{\!(\msI_1,\,\cdot)}&=\bigl\{ d^{\around{n}}_{(h,\,\cdot)} : h\in\msI_1 \bigr\}.
\end{align*}
\end{prop}
\vspace*{1mm}
\begin{proof}
Let  $x,y\in\mD_n$, $f\in\msA$, $g\in\msM_1$, and $h\in\msI_1$.\\
By assumption, there exist $x,y\in\mD_n>1$ with $\gcd(x,y)=1$ and $x\cdot y\in\mD_n$.

Moreover, we have\\
$d^{\around{n}}_{(f,+)}(1,x)=f(x)$, \quad$d^{\around{n}}_{(g,+)}(1,x)=g(x)-1$, \quad$d^{\around{n}}_{(h,+)}(1,x)=1-h(x)$, \quad and\\
$d^{\around{n}}_{(h,\,\cdot)}(1,x)=1-h(x)$. For $x>1$, we get $g(x)-1>0$ and $1-h(x)>0$.\\[-0.5ex]

We prove the disjointness separately assuming a common element.\\[-0.5ex]

$\msD^{\!\around{n}}_{\!(\msA,+)} \cap \msD^{\!\around{n}}_{\!(\msM_1,+)} = \emptyset$.\\
Suppose there are $f\in\msA$ and $g\in\msM_1$ such that $d^{\around{n}}_{(f,+)}=d^{\around{n}}_{(g,+)}$. So also\\
$d^{\around{n}}_{(f,+)}(1,x)=f(x)=g(x)-1=d^{\around{n}}_{(g,+)}(1,x)$ and $f(x)=g(x)-1$.\\
Then $f(x\cdot y)=g(x\cdot y)-1=g(x)\cdot g(y)-1$, \quad and\\[0.5ex]
$f(x\cdot y)=f(x)+f(y)=g(x)+g(y)-2=g(x\cdot y)-1-\bigl(g(x)-1\bigr)\cdot\bigl(g(y)-1\bigr)$.\\[0.5ex]
Therefore, $\bigl(g(x)-1\bigr)\cdot\bigl(g(y)-1\bigr)=0$, \ i.e. $x=1$ or $y=1$ in contradiction to $x,y>1$.\\[-0.5ex]

$\msD^{\!\around{n}}_{\!(\msA,+)} \cap \msD^{\!\around{n}}_{\!(\msI_1,+)} = \emptyset$.\\
From $d^{\around{n}}_{(f,+)}(1,x)=f(x)=1-h(x)=d^{\around{n}}_{(h,\,\cdot)}(1,x)$ follows $f(x)=1-h(x)$.\\
Then $f(x\cdot y)=1-h(x\cdot y)=1-h(x)\cdot h(y)$, \quad and also\\[0.5ex]
$f(x\cdot y)=f(x)+f(y)=2-h(x)-h(y)=\bigl(1-h(x)\bigr)\cdot \bigl(1-h(y)\bigr)+1-h(x)\cdot h(y)$.\\[0.5ex]
So, \ $\bigl(1-h(x)\bigr)\cdot \bigl(1-h(y)\bigr)=0$ in contradiction to $x,y>1$.\\[-0.5ex]

$\msD^{\!\around{n}}_{\!(\msA,+)} \cap \msD^{\!\around{n}}_{\!(\msI_1,\,\cdot)} = \emptyset$.\\
Analogous to the last case, since $\msD^{\!\around{n}}_{\!(\msI_1,\,\cdot)}(1,x)=\msD^{\!\around{n}}_{\!(\msI_1,+)}(1,x)$. So, we get again\\
$f(x,y)=1-h(x)\cdot h(y)=2-h(x)-h(y)$, \\
and finally $\bigl(1-h(x)\bigr)\cdot \bigl(1-h(y)\bigr)=0$, contradiction.\\[-0.5ex]

$\msD^{\!\around{n}}_{\!(\msM_1,+)} \cap \msD^{\!\around{n}}_{\!(\msI_1,+)} = \emptyset$.\\
From $d^{\around{n}}_{(g,+)}(1,x)=g(x)-1=1-h(x)=d^{\around{n}}_{(h,+)}(1,x)$ follows $g(x)=2-h(x)$.\\
Then $g(x\cdot y)=2-h(x\cdot y)=2-h(x)\cdot h(y)$, \quad and also\\[0.5ex]
$g(x\cdot y)=g(x)\cdot g(y)=\bigl(2-h(x)\bigr)\cdot \bigl(2-h(y)\bigr)=4-2\cdot h(x)-2\cdot h(y)+h(x)\cdot h(y)$.\\[0.5ex]
So, \ $2-h(x)\cdot h(y)=4-2\cdot h(x)-2\cdot h(y)+h(x)\cdot h(y)$ \ or \ $h(x)+h(y)=1+h(x\cdot y)$,\linebreak\\[-2ex]
i.e. $\bigl(1-h(x)\bigr)\cdot \bigl(1-h(y)\bigr)=0$, contradiction.\\[-0.5ex]

$\msD^{\!\around{n}}_{\!(\msM_1,+)} \cap \msD^{\!\around{n}}_{\!(\msI_1,\,\cdot)} = \emptyset$.\\
Analogous to the last case, since $\msD^{\!\around{n}}_{\!(\msI,\,\cdot)}(1,x)=\msD^{\!\around{n}}_{\!(\msI,+)}(1,x)$. So, we get again\\
$2-h(x)\cdot h(y)=g(x\cdot y)=4-2\cdot h(x)-2\cdot h(y)+h(x)\cdot h(y)$,\\[0.5ex]
and finally $\bigl(1-h(x)\bigr)\cdot \bigl(1-h(y)\bigr)=0$, contradiction.\\[-0.5ex]

$\msD^{\!\around{n}}_{\!(\msI_1,+)} \cap \msD^{\!\around{n}}_{\!(\msI_1,\,\cdot)} = \emptyset$.\\
$d^{\around{n}}_{(h,+)}(x,y)=2-h(x)-h(y)=1-h(x)\cdot h(y)=d^{\around{n}}_{(h,\,\cdot)}(x,y)$ \\
also results in the requirement $\bigl(1-h(x)\bigr)\cdot \bigl(1-h(y)\bigr)=0$, contradiction.
\end{proof}

\vspace*{1mm}
Although the sets of pseudometrics considered above are disjoint, the morphisms of Proposition~\ref{iso-binary} between $\msA$, $\msM$, and $\msI_0$ preserving the binary operations can be\lb applied to $\msD^{\!\around{n}}_{\!(\msA,+)}$, $\msD^{\!\around{n}}_{\!(\msM_1,+)}$, $\msD^{\!\around{n}}_{\!(\msI_0,+)}$, and $\msD^{\!\around{n}}_{\!(\msI_0,\,\cdot)}$ because of the canonical assignments from the respective arithmetic function to the corresponding pseudometric.\\ A semigroup-structure can be transferred in this way.

However, for admissible additive arithmetic functions, the canonical assignment constitutes a natural homomorphism. This does not apply to multiplicative functions because of the other binary operation.\\[-3mm]

\begin{lemm} \label{homomorphic}
Let $n\in\mN>1$. Then $f\in\msA \to d^{\around{n}}_{(f,+)} \in \msD^{\!\around{n}}_{\!(\msA,+)}$ represents a semigroup homomorphism between $(\msA,+)$ and $\bigl(\msD^{\!\around{n}}_{\!(\msA,+)},+\bigr)$.
\end{lemm}
\ \\
\begin{proof}
Let $x,y\in\mD_n$ and $f_1,f_2\in\msA$.\\[0.5ex]
We get $(f_1+f_2)(x)=f_1(x)+f_2(x)$ as well as $d^{\around{n}}_{(f_1,+)}(x,y)=f_1(x)+f_1(y)$ and\lb analogously $d^{\around{n}}_{(f_2,+)}(x,y)=f_2(x)+f_2(y)$ by definition. Therefore,
\vspace*{-3mm}
\begin{align*}
d^{\around{n}}_{(f_1+f_2,+)}(x,y)&=(f_1+f_2)(x)+(f_1+f_2)(y) \\
&=f_1(x)+f_2(x)+f_1(y)+f_2(y) = d^{\around{n}}_{(f_1,+)}(x,y)+d^{\around{n}}_{(f_2,+)}(x,y).
\end{align*}
\vspace*{-14mm}\\
\end{proof}
\ \\\\

Finally, we add another option to the list of how pseudometrics can be generated from admissible arithmetic functions. Since the focus should be on divisibility\lb properties, we reduce the distances generated from functions of $\msA$ and $\msI$ to that of coprime arguments. This also leads to $1$-centred pseudometrics as before, but the distances between large integers can become smaller if they have common divisors. Proposition~\ref{disjoint_2}~and~Lemma~\ref{homomorphic} can be transferred to them accordingly.

Further exploration of these highly interesting divisibility metrics is beyond the scope of this paper. The derived preorders, relations and partitions remain the same as those of the pseudometrics defined above. All of them are solely determined by the distances from $1$.

\vspace*{6mm}
\begin{theo} \label{NTF_metric}
Let $n\in\mN>1$,  $x,y\in\mZ$, $f\in\msA$, and $h\in\msI$. Then
\vspace*{-3mm}
\begin{align*}
\delta^{\around{n}}_{(f,+)}(x,y)&=\begin{cases}
   0 &\hspace*{7mm}\text{for $x=y$,}\\
   f\left(\fcn xxy\right)+f\left(\fcn yxy\right) &\hspace*{7mm}\text{for $x\ne y$,} \end{cases}\\
\delta^{\around{n}}_{(h,+)}(x,y)&=\begin{cases}
   0 &\text{for $x=y$,}\\
   2-h\left(\fcn xxy\right)-h\left(\fcn yxy\right) &\text{for $x\ne y$, \quad and} \end{cases}\\
\delta^{\around{n}}_{(h,\,\cdot\,)}\hspace{0.6mm}(x,y)&=\begin{cases}
   0 &\hspace*{2mm}\text{for $x=y$,}\\
   1-h\left(\fcn xxy\right)\cdot h\left(\fcn yxy\right) &\hspace*{2mm}\text{for $x\ne y$} \end{cases}
\end{align*} \vspace*{-3mm}\\
are pseudometrics on $\mZ$, so also on $\mD_n$.

If $f\in\msA_0$ and $h\in\msI_1$ then they are metrics on $\mD_n$.
\end{theo}

\vspace*{3mm}
\begin{proof}
In all cases, identity and symmetry follow by construction. We prove the three triangle inequalities on $\mD_n$ separately. The corresponding extensions of the\lb pseudometrics to pseudometrics on $\mZ$ result according to Lemma~\refs{extended_pseudometric}.
\pagebreak

Let $x,y,z\in\mZ$. We decompose  $x$, $y$, and $z$ into seven coprime integers as follows.
{\small
\begin{align*}
f_{xyz}&=\gcd(x,y,z) \hspace{47.3mm}\text{, i.e. }f_{xyz}\mdiv x \land f_{xyz}\mdiv y \land f_{xyz}\mdiv z.\\
f_{xy}&=\frac{\gcd(x,y)}{f_{xyz}}=\frac{\gcd(x,y)}{\gcd(x,y,z)}\hspace{25mm}\text{, i.e. }f_{xy}\mdiv x \land f_{xy}\mdiv y \land f_{xy}\ndiv z,\\
f_{yz},f_{xz} &=\dots\text{ analogously}.\\
f_{x}&=\frac{x}{f_{xy}\cdot f_{xz}\cdot f_{xyz}}=\frac{x\cdot\gcd(x,y,z)}{\gcd(x,y)\cdot\gcd(x,z)}\hspace{4.5mm}\text{, i.e. }f_{x}\mdiv x \land f_{x}\ndiv y \land f_{x}\ndiv z,\hspace{5.6mm}\\
f_{y},f_{z} &=\dots\text{ analogously}.\\
\end{align*} \vspace*{-12mm}\\
}
Thus, $x=f_{x}\cdot f_{xy}\cdot f_{xz}\cdot f_{xyz}$ and accordingly for $y$ and $z$.\\

\vspace*{3mm}
For use in the coming derivations we first prove some general divisibility assertions for integers. \vspace*{-10mm}\\

{\small
\begin{flalign*}
\text{(a)} \hspace*{16mm} \qquad \fcd {\phantom{y}x\phantom{y}}xy &\perp \fcd yyz.\\
\fcd {\phantom{y}x\phantom{y}}xy &=\frac{x}{f_{xy}\cdot f_{xyz}}=f_{x}\cdot f_{xz}  \text{, \quad whereas}\\
\fcd yyz &=\frac{y}{f_{yz}\cdot f_{yxz}}=f_{y}\cdot f_{yx}=f_{y}\cdot f_{xy}. \quad \text{All $f_. $ are pairwise coprime.}&&
\end{flalign*}\vspace*{-13mm}

\begin{flalign*}
\text{(b)} \hspace*{5mm} (\gcd(x,y)\cdot\gcd(y,z))&\ \,\mdiv\ \,(y\cdot\gcd(x,z))\text{, \quad i.e. }\pcd y{\gcd(x,z)}xyyz\in\mZ.\\
y\cdot\gcd(x,z) &=f_{y}\cdot f_{xy}\cdot f_{yz}\cdot f_{xyz}\cdot f_{xz}\cdot f_{xyz}\\
&=f_{xy}\cdot f_{xyz}\cdot f_{yz}\cdot f_{xyz}\ \cdot\ f_{y}\cdot f_{xz}\text{, \quad and}\\
\gcd(x,y)\cdot\gcd(y,z) &=f_{xy}\cdot f_{xyz}\cdot f_{yz}\cdot f_{xyz}.&&
\end{flalign*}\vspace*{-13mm}

\begin{flalign*}
\text{(c)} \hspace*{24mm} \fcd {\phantom{y}x\phantom{y}}xz &\perp \pcd y{\gcd(x,z)}xyyz.\text{\qquad Given (b).}\\
\fcd {\phantom{y}x\phantom{y}}xz &=f_{x}\cdot f_{xy}\text{, \quad whereas}\\
\pcd y{\gcd(x,z)}xyyz &=f_{y}\cdot f_{xz}.&&
\end{flalign*}\vspace*{-11mm}

\begin{flalign*}
\text{(d)} \hspace*{5.4mm} \pcd x{\gcd(y,z)}xyxz &\perp \pcd y{\gcd(x,z)}xyyz.\text{\qquad Given (b).}\\
\pcd x{\gcd(y,z)}xyxz &=f_{x}\cdot f_{yz}\text{, \quad whereas}\\
\pcd y{\gcd(x,z)}xyyz &=f_{y}\cdot f_{xz}.&&
\end{flalign*}\vspace*{-3mm}

Furthermore, we note the following inequality for $h\in\msI$. \vspace*{-11mm}\\

\begin{flalign*}
\text{(e)} \hspace*{13.5mm} 2-h(x)-h(y)&\ge1-h(x)\cdot h(y)\text{\qquad because}\\
2-h(x)-h(y)&\;\!-\;\!\bigl(1-h(x)\cdot h(y)\bigr)=1-h(x)-h(y)+h(x)\cdot h(y)=\\
&\hspace*{33.1mm} =\bigl(1-h(x)\bigr)\cdot \bigl(1-h(y)\bigr)\ge0.&&
\end{flalign*}
}%

Proof of the triangle inequalities. W.l.o.g. let now $x,y,z\in\mD_n$.
\vspace*{2mm}
{\footnotesize
\begin{flalign*}
\delta^{\around{n}}_{(f,+)} : \ &\phantom{\ =\ }\delta^{\around{n}}_{(f,+)}(x,y) \ +\ \delta^{\around{n}}_{(f,+)}(y,z)=\\
&=f\left(\fcd xxy\right)+f\left(\fcd yxy\right) \ +\  f\left(\fcd yyz\right)+f\left(\fcd zyz\right)\\
&=f\left(\fcd xxy\right)+f\left(\fcd yyz\right) \ +\  f\left(\fcd yxy\right)+f\left(\fcd zyz\right)\\
&\underset{\text{(a)}}{=}f\left(\pcd xyxyyz\right) \ +\  f\left(\pcd yzxyyz\right)\\
&=f\left(\fcd xxz\mdot\pcd y{\gcd(x,z)}xyyz\right) \ +\  f\left(\fcd zxz\mdot\pcd y{\gcd(x,z)}xyyz\right)\\
&\underset{\text{(c)}}{=}f\left(\fcd xxz\right) + f\left(\fcd zxz\right) \ +\ 2\mdot f\left(\pcd y{\gcd(x,z)}xyyz\right)\\
&\ge f\left(\fcd xxz\right) + f\left(\fcd zxz\right)=\delta^{\around{n}}_{(f,+)}(x,z).&&
\end{flalign*}\vspace*{-8mm}

\begin{flalign*}
\delta^{\around{n}}_{(h,+)} : \ &\phantom{\ =\ }2-h\left(\fcd xxy\right)-h\left(\fcd yxy\right) \ +\  2-h\left(\fcd yyz\right)-h\left(\fcd zyz\right)=\\
&=2-h\left(\fcd xxy\right)-h\left(\fcd yyz\right) \ +\  2-h\left(\fcd yxy\right)-h\left(\fcd zyz\right)\\
&\underset{\text{(e)}}{\ge} 1-h\left(\fcd xxy\right)\cdot h\left(\fcd yyz\right) \ +\  1-h\left(\fcd yxy\right)\cdot h\left(\fcd zyz\right)\\
&\underset{\text{(a)}}{=}1-h\left(\pcd xyxyyz\right) \ +\  1-h\left(\pcd yzxyyz\right)\\
&=1-h\left(\fcd xxz\mdot\pcd y{\gcd(x,z)}xyyz\right) \ +\  1-h\left(\fcd zxz\mdot\pcd y{\gcd(x,z)}xyyz\right)\\
&\underset{\text{(c)}}{=}2-\Biggl(h\left(\fcd xxz\right) + h\left(\fcd zxz\right)\Biggr)\cdot h\left(\pcd y{\gcd(x,z)}xyyz\right)\\
&\ge 2-h\left(\fcd xxz\right) + h\left(\fcd zxz\right).&&
\end{flalign*}\vspace*{-8mm}

\begin{flalign*}
\delta^{\around{n}}_{(h,\,\cdot\,)} : \ &\phantom{\ =\ }1-h\left(\fcd xxy\right)\cdot h\left(\fcd yxy\right) \ +\  1-h\left(\fcd yyz\right)\cdot h\left(\fcd zyz\right)=\\
&\underset{\text{(a)}}{=}1-h\left(\pcd xyxyxy\right) \ +\  1-h\left(\pcd yzyzyz\right)\\
&=2-h\left(\pcd x{\gcd(y,z)}xyxz \cdot \pcd y{\gcd(x,z)}xyyz\right)\\
	&\qquad -h\left(\pcd z{\gcd(x,y)}xzyz \cdot \pcd y{\gcd(x,z)}xyyz\right)\\
&\underset{\text{(d)}}{=}2-\Biggl(h\left(\pcd x{\gcd(y,z)}xyxz\right)+h\left(\pcd z{\gcd(x,y)}xzyz\right)\Biggr) \cdot h\left(\pcd y{\gcd(x,z)}xyyz\right) \\
&\ge 2-h\left(\pcd x{\gcd(y,z)}xyxz\right)-h\left(\pcd z{\gcd(x,y)}xzyz\right) \\
&\underset{\text{(e)}}{\ge} 1-h\left(\pcd x{\gcd(y,z)}xyxz\right) \cdot h\left(\pcd z{\gcd(x,y)}xzyz\right)&&
\end{flalign*}

\pagebreak

\vspace*{-8mm}
\begin{flalign*}
\phantom{\delta^{\around{n}}_{(h,\,\cdot\,)} : \ }&\underset{\text{(d)}}{\ge}1-h\left(\pcd x{\gcd(y,z)}xyxz \cdot \pcd z{\gcd(x,y)}xzyz\right)=1-h\left(\pcd xzxzxz\right) \\
&\underset{\text{(a)}}{=}1-h\left(\fcd xxz\right)\cdot h\left(\fcd zxz\right).&&\\
\end{flalign*}\vspace*{-22mm}\\
}
\end{proof}

\vspace*{4mm}
\begin{rema}
Arithmetic functions $g\in\msM$ do not generally lead to pseudometrics by\lb using coprime additive combination. The following counterexample violates the\lb triangle inequality.
Let $n=24$, $x=2$, $y=6$, $z=24$, and $g(x)=x$. Then
\vspace*{-2mm}
{\small
\begin{align*}
\delta^{\around{n}}_{(g,+)}(x,y)+\delta^{\around{n}}_{(g,+)}(y,z)&=g(1)+g(3)-2+g(1)+g(4)-2=5\\[-0.5ex]
&\le11=g(1)+g(12)-2=\delta^{\around{n}}_{(g,+)}(x,z).
\end{align*}
}
\end{rema}

\ \\\\[-3mm]


\section{Arithmetic functions and preorders\phj} \label{AFO}

In the last section, we examined pseudometrics on $\mZ$ generated from admissible\lb arithmetic functions. Such pseudometrics can generally be used to derive preorders and corresponding equivalences and partitions, as described in Sections~\ref{MO}~and~\refs{EP} .

We now start investigating preorders related to functions defined in Theorems \ref{coro_metric2} and \refs{coro_metric3}. Because of Lemma~\refs{extended_pseudometric}, it is again sufficient to only consider values of the respective arithmetic function for all factors of $n$.

\vspace*{1mm}
\begin{theo} \label{equivalent_relation}
Let $n\in\mN>1$, $f\in\msA$, $g=\e^{\ f}\in\msM$, and $h=\e^{-f}\in\msI_0$.
The preorders induced by $d^{\around{n}}_{(f,+)}$, $d^{\around{n}}_{(g,+)}$, $d^{\around{n}}_{(h,+)}$, and $d^{\around{n}}_{(h,\,\cdot)}$ are equivalent on $\mZ$, i.e.
\vspace*{-3mm}
\begin{align*}
x\preceq_{d^{\around{n}}_{(f,+)}} y \iff x\preceq_{d^{\around{n}}_{(g,+)}} y &\iff x\preceq_{d^{\around{n}}_{(h,+)}} y \iff x\preceq_{d^{\around{n}}_{(h,\,\cdot)}} y.\\
\text{Furthermore,}\hspace{31mm} x\preceq_{d^{\around{n}}_{(h_0,+)}} y &\iff x\preceq_{d^{\around{n}}_{(h_0,\,\cdot)}} y \qquad \text{for every $h_0\in\msI\setminus\msI_0$.}
\end{align*}
\vspace*{-13mm}\\
\end{theo}

\vspace*{1mm}
\begin{proof}
According to Lemma~\refs{extended_pseudometric}, it is sufficient to prove the assertions on $\mD_n$. Let $x,y\in\mD_n$. By Lemma~\refs{induced_preorder}, we have
\vspace*{-3mm}
\begin{align*}
x\preceq_{d^{\around{n}}_{(f,+)}} y &\iff f(x)\le f(y),\\
x\preceq_{d^{\around{n}}_{(g,+)}} y &\iff g(x)-1\le g(y)-1,\\
x\preceq_{d^{\around{n}}_{(h,+)}} y &\iff 1-h(x)\le 1-h(y),\\
x\preceq_{d^{\around{n}}_{(h,\,\cdot)}} y &\iff 1-h(x)\le 1-h(y).
\end{align*}\vspace*{-9mm}\\

The assertion follows from the bijective functions $g=\e^{\ f} \text{and } h=\e^{-f}$.

\vspace*{-6mm}
\begin{align*}
g(x)-1 \le g(y)-1 &\iff \e^{\ f(x)}-1 \le \e^{\ f(y)}-1 \hspace*{5mm}\iff f(x)\le f(y).\\
1-h(x) \le 1-h(y) &\iff 1-\e^{-f(x)} \le 1-\e^{-f(y)} \iff f(x)\le f(y).
\end{align*}
\newpage

\vspace*{-12mm}
\begin{align*}
\text{For } h_0\in\msI\setminus\msI_0 \text{, we get}\hspace*{3mm}&&\hspace*{40mm}\\
x\preceq_{d^{\around{n}}_{(h_0,+)}} y &\iff 1-h_0(x)\le 1-h_0(y) \hspace*{6mm}\iff x\preceq_{d^{\around{n}}_{(h_0,\,\cdot)}} y.
\end{align*}\vspace*{-14mm}\\
\end{proof}
\vspace*{0mm}

\begin{rema}
The isomorphisms between $\msA$, $\msM$, and $\msI_0$ of Proposition~\ref{iso-binary} are therefore also order isomorphisms with respect to the induced preorders on $\mZ$.\\
\end{rema}

As a provisional interim conclusion depicted in Figure~\refs{dia3}, we state the following. There are four pseudometrics on $\mZ$ that we can derive from any triple of admissible arithmetic functions $(f,g,h)$ with $f\in\msA$, $g=\e^{\ f}\in\msM$, and $h=\e^{-f}\in\msI_0$. However, the corresponding induced preorders coincide.

Otherwise, each $h_0\in\msI\setminus\msI_0$ can be used to create two different pseudometrics\lb on $\mZ$. The corresponding induced preorders of both pseudometrics also coincide.\lb In this case, however, the characterisation of possible preorders is more complicated.

For a given $n\in\mN$, all such preorders derived from any $h\in\msI_0$ can also be derived from a suitable $h_0\in\msI\setminus\msI_0$. On the other hand, there are $h_0\in\msI\setminus\msI_0$ such that the respective preorder cannot be induced from a pseudometric derived from any $h\in\msI_0$. We will prove the details in Proposition~\ref{zero} below.\\\\

\begin{figure}[H]
  \centering\vspace*{-3mm}
\begin{gather*}\xymatrix{
  f\in\msA	\ar[rr]^{\ref{coro_metric2}}
  	\ar@{<->}[d]_{\ref{iso-binary}}
 		& & d^{\around{n}}_{(f,+)}
  	    \ar[ddrr]_{\ref{equivalent_relation}}
 	        & \hspace{17mm} \\
  g=\e^{\ f}\in\msM \ar[rr]_{\ref{coro_metric2}}
  	\ar@{<->}[d]_{\ref{iso-binary}}
  		& & d^{\around{n}}_{(g,+)}
  	    \ar[drr]_{\ref{equivalent_relation}} \\
  h=\e^{-f}\in\msI_0 \ar[rr]_{\ref{coro_metric2}}
	 \ar[drr]_{\ref{coro_metric3}}
  		& & d^{\around{n}}_{(h,+)}
  	    \ar[rr]_{\ref{equivalent_relation}}
	        & & \preceq_{d^{\around{n}}_{(h,+)}} \\
  & & d^{\around{n}}_{(h,\,\cdot)}
  	    \ar[urr]_{\ref{equivalent_relation}} \\
  h_0\in\msI\setminus\msI_0 \ar[rr]_{\ref{coro_metric2}}
	 \ar[drr]_{\ref{coro_metric3}} 
  		& & d^{\around{n}}_{(h_0,+)}
  	    \ar[rr]_{\ref{equivalent_relation}}
	        & & \preceq_{d^{\around{n}}_{(h_0,+)}} \\
  & & d^{\around{n}}_{(h_0,\,\cdot)}
  	    \ar[urr]_{\ref{equivalent_relation}}
}\end{gather*}
   \caption{Relationships between admissible functions, induced pseudometrics, and corresponding preorders.}\label{dia3}\ 
\end{figure}

We recall and emphasise the existence of triples of admissible arithmetic functions that lead to the same preorder for a given $n$.

\begin{coro} \label{triple}
Let $n\in\mN>1$. For each $h\in\msI_0$, there exist admissible arithmetic functions $f\in\msA$ and $g\in\msM$ such that the following preorders coincide on $\mZ$.
\vspace*{-2mm}
\begin{align*}
x\preceq_{d^{\around{n}}_{(f,+)}} y \iff x\preceq_{d^{\around{n}}_{(g,+)}} y &\iff x\preceq_{d^{\around{n}}_{(h,+)}} y \iff x\preceq_{d^{\around{n}}_{(h,\,\cdot)}} y.
\end{align*}
\vspace*{-12mm} \\
\end{coro}

\begin{proof}
The functions $f=-\ln(h)$ and $g=1/h$ meet the requirements according to Theorem~\refs{equivalent_relation}.
\end{proof}\ \\

The remaining central question is what kinds of preorders of the considered form exist at all. According to the previous corollary, all these preorders can be induced from pseudometrics generated from an admissible interval-bounded multiplicative arithmetic function using additive combination. We therefore limit ourselves to the connection
\vspace*{-2mm}
\begin{gather*}\xymatrix{
  h\in\msI \ar[rr]	& & d^{\around{n}}_{(h,+)}\in\msD^{\!\around{n}}_{\!(\msI,+)}
  	    \ar[rr]  & & \preceq_{d^{\around{n}}_{(h,+)}}.
}\end{gather*}
\vspace*{-5mm}

As already mentioned, according to Lemma~\ref{extended_pseudometric}~and~Corollary~\refs{coro_metric1}, every metric\lb induced by an admissible arithmetic function is uniquely defined by the values\lb of the respective function for all divisors of $n$. And so is the induced preorder.

In fact, every multiplicative or additive arithmetic function is already determined by its values for prime powers. Each metric examined is thus also uniquely defined by the function values for all prime powers dividing n. However, different functions can generate the same pseudometric, and in turn different pseudometrics can induce the same preorder.

\begin{lemm}
Let $n\in\mN>1$, $x,y\in\mP^*_n$, and $h_1,h_2\in\msI$ with $h_1(z)=h_2(z)$ for all $z\in\mP^*_n$. Then
\vspace*{-3mm}
\begin{align*}
d^{\around{n}}_{(h_1,+)}(x,y) &\ \ \ =\ \ \ \ d^{\around{n}}_{(h_2,+)}(x,y), \quad \text{and}\\
x\preceq_{d^{\around{n}}_{(h_1,+)}} y &\iff x\preceq_{d^{\around{n}}_{(h_2,+)}} y .
\end{align*}
\vspace*{-12mm} \\
\end{lemm}
\vspace*{1mm}
\begin{proof}

For every $x=\prod_{i=1}^\infty p_i^{\alpha_i}\in\mD_n$ and any multiplicative function $h\in\msI$, we have $h(x)=\prod_{i=1}^\infty h(p_i^{\alpha_i})$, $p_i\in\mP$, $\alpha_i\in\mN_0$.

\vspace*{2mm}
The first assertion follows from the definitions in Theorem~\refs{coro_metric2}.\\[-2.0ex]

For every $h\in\msI$, we have according to Lemma~\ref{induced_preorder}
\vspace*{-3mm}
\begin{align*}
x\preceq_{d^{\around{n}}_{(h,+)}} y &\iff 1-h(x)\le 1-h(y),\quad\text{ i.e.}\\[-0.5ex]
\Bigl( x\preceq_{d^{\around{n}}_{(h_1,+)}} y \iff x\preceq_{d^{\around{n}}_{(h_2,+)}} y \Bigr) &\iff \Bigl( h_1(x)\ge h_1(y) \iff h_2(x)\ge h_2(y) \Bigr) .
\end{align*}
\vspace*{-5mm} \\
The right-hand side of the equivalence holds by assumption.
\end{proof}

\begin{rema}
According to the previous lemma and the definitions in Theorem~\refs{coro_metric2},\lb two functions that differ on $\mP^*_n$ lead to different pseudometrics. However, two such pseudometrics can result in the same preorder even if their values for $\mP^*_n\times\mP^*_n$\lb are not equal, e.g.\\
$h_1(2)=0.3,\; h_1(3)=0.6,\; h_1(6)=0.18$ \quad and \quad $h_2(2)=0.2,\; h_2(3)=0.5,\; h_2(6)=0.1$ are examples for $n=6$.\\\\
\end{rema}

As a consequence of the last lemma, a function $h\in\msI$ can only vanish at an $x\in\mD_n$\lb if there is a prime power $y \mdiv x$ with $h(y)=0$. Hence, for every $n\in\mN>1$, we\lb introduce a customised partition of the set $\msI$  that makes the induced preorders\lb disjoint. The distinction between the fixed sets $\msI$ and $\msI\setminus\msI_0$ can lead to overlapping sets of pre-orders for different $n$.

\begin{defi}
Let $n\in\mN>1$. We define
\vspace*{-3mm}
\begin{eqnarray*}
\msI^{\around{n}}_{>0} &=& \{ h\in\msI \;:\; \forall x\in\mD_n : h(x)>0 \}, \quad \text{and}\\
\msI^{\around{n}}_{=0} &=& \{ h\in\msI \;:\; \exists x\in\mD_n : h(x)=0 \}.
\end{eqnarray*}
\vspace*{-13mm}\\
\end{defi}

\begin{rema}
The definition is equivalent to
\vspace*{-3mm}
\begin{eqnarray*}
\msI^{\around{n}}_{>0}     &=& \{ h\in\msI \;:\; \forall x\in\mP^*_n : h(x)>0 \}, \quad \text{and}\\
\msI^{\around{n}}_{=0} &=& \{ h\in\msI \;:\; \exists x\in\mP^*_n : h(x)=0 \}.
\end{eqnarray*}
\vspace*{-12mm}\\

The relationship between the two partitioning types of $\msI$ can be expressed \vspace*{1mm}in other words. While $\msI_0$ is a subset of each $\msI^{\around{n}}_{>0}$, each $\msI^{\around{n}}_{=0}$ is a subset of $\msI\setminus\msI_0$.
\vspace*{-2mm}
\begin{eqnarray*}
\msI_0      &=& \bigcap\limits_{n\in\mN} \msI^{\around{n}}_{>0}\ ,\\
\msI\setminus\msI_0 &=& \bigcup\limits_{n\in\mN} \msI^{\around{n}}_{=0}\ .
\end{eqnarray*}
\vspace*{-5mm}
\end{rema}

\begin{prop} \label{zero}
Let $n\in\mN>1$. 

For every $h\in\msI^{\around{n}}_{>0}$ there exist functions $f\in\msA_0$, $g\in\msM_1$, $h^*_1\in\msI_0$, and $h^*_0\in\msI\setminus\msI_0$ inducing the same pseudometric and the same preorder on $\mZ$.

However, for every $h\in\msI^{\around{n}}_{=0}$ there exist only $h^*_0\in\msI\setminus\msI_0$ inducing the same pseudometric and the same preorder on $\mZ$.
\end{prop}
\begin{proof}
Let $\mP^*_n$ be the set of prime powers dividing $n$.

For $h\in\msI^{\around{n}}_{>0}$, we set
\vspace*{-9mm}
\begin{eqnarray*}
f(x)          &=& -\ln(h(x)),\\
g(x)         &=& 1/(h(x)),\\
h^*_1(x) &=& h(x),\quad\text{  and}\\
h^*_0(x) &=& h(x)
\end{eqnarray*}
\vspace*{-8mm}\\
for all $x\in\mP^*_n$. These functions meet the requirements if their values for all prime powers y not dividing $n$ were chosen arbitrarily but satisfy
\vspace*{-3mm}
\begin{eqnarray*}
&f(y)&>0,\\
&g(y)&>1,\\
0< &h^*_1(y)&\le1,\quad\text{ and}\\
0\le &h^*_0(y)&\le1 \quad\land\quad h^*_0(y)=0\text{ \ for at least one such }y.
\end{eqnarray*}
\vspace*{-7mm}

The functions $f$, $g$, and $h^*_1$ form a triple as described in Corollary~\refs{triple}. This follows from $\msI_0\subseteq\msI^{\around{n}}_{>0}$.\\[-0.5ex]

For $h\in\msI^{\around{n}}_{=0}\subseteq\msI\setminus\msI_0$\,, we already have $h\in\msI\setminus\msI_0$, and $h^*_0=h$ is sufficient. However, $h^*_0$ may also have additional zero function values for prime powers y not dividing $n$.

Pseudometrics and preorders are uniquely determined by the function values\lb on $\mP^*_n$. There can be no appropriate $f\in\msA_0$, $g\in\msM_1$, or $h^*_1\in\msI_0$ because $h$ always contains a zero value on $\mP^*_n$.
\end{proof}
\ \\

For every pseudometric $d^{\around{n}}_{(h,+)}\in\msD^{\!\around{n}}_{\!(\msI,+)}$, we know $x \sim_{d^{\around{n}}_{(h,+)}} y \ \Longrightarrow\  x \sim_{\preceq_{d^{\around{n}}_{(h,+)}}} y$, \ i.e.\linebreak
\[\mD_n\qs\sim_{d^{\around{n}}_{(h,+)}} \ \simeq\  \mZ\qs\sim_{d^{\around{n}}_{(h,+)}} \quad\sqsubseteq\quad \mZ\qs\sim_{\preceq_{d^{\around{n}}_{(h,+)}}}\ \simeq\  \mD_n\qs\sim_{\preceq_{d^{\around{n}}_{(h,+)}}},\] \vspace*{-2mm}\\
by Propositions~\ref{extension_isomorphism}~and~\refs{refinement}.

If furthermore $h\in\msI_1$, \ $d^{\around{n}}_{(h,+)}$ is a metric on $\mD_n$ according to Theorem~\ref{coro_metric2}, \ i.e. $d^{\around{n}}_{(h,+)}(x,y)=0\Longrightarrow x=y$. Then $\mD_n \ \simeq\  \mD_n\qs\sim_{d^{\around{n}}}$ by Lemma~\ref{metric_isomorphism} and $\mD_n \ \simeq\  \mZ\qs\sim_{d^{\around{n}}}$ by Proposition~\refs{extension_isomorphism}. Each element of $\mD_n$ builds an extra equivalence class modulo $\sim_{d^{\around{n}}}$.

Therefore, we focus on the characterisation of the quotient sets $\mD_n\qs\sim_{\preceq_{d^{\around{n}}_{(h,+)}}}$ with\lb respect to the induced preorders for $h\in\msI$. So, because of the isomorphism\lb with $\mZ\qs\sim_{\preceq_{d^{\around{n}}_{(h,+)}}}$, we can thus study a more differentiated partitioning of $\mZ$ in terms of divisibility properties.\\

These quotient sets aggregate integers that are not distinguishable by the respective preorder. They are invariant against permutation of the function values of $h$ on $\mP^*_n$. All function values of $h$ on $\mD_n$ are determined by it.

\begin{lemm} \label{permutation1}
Let $n\in\mN>1$, $h_1,h_2\in\msI$, and $h_2(x)=\eta\bigl(h_1(x)\bigr)$ for all $x\in\mP^*_n$ where $\eta$ is a permutation of $h_1(\mP^*_n)$. Then there exists a uniquely determined extension $\nu$ of $\eta$ on $h_1(\mD_n)$ such that $h_2(y)=\nu\bigl(h_1(y)\bigr)$ for all $y\in\mD_n$.
\end{lemm}

\begin{proof}
By assumption, $\nu(h_1(x))=\eta\bigl(h_1(x)\bigr)$ for all $x\in\mP^*_n$.\\
Furthermore, $\nu\bigl(h_1(1)\bigr)=h_1(1)=h_2(1)=1$ is mandatory.

For $y\in\mD_n\setminus\mP^*_n>1$, there must exist $t_1,t_2\in\mP^*_n>1$ such that $y=t_1\cdot t_2$ and $t_1\perp t_2$. Then
\vspace*{-5mm}
\begin{align*}
\nu\bigl(h_1(y)\bigr)&=\nu\bigl(h_1(t_1\cdot t_2)\bigr)=\nu\bigl(h_1(t_1)\cdot h_1(t_2)\bigr)\\
&\underset{Df}{=}\nu\bigl(h_1(t_1)\bigr)\cdot \nu\bigl(h_1(t_2)\bigr)=h_2(t_1)\cdot h_2(t_2)=h_2(t_1\cdot t_2)=h_2(y)
\end{align*}
\vspace*{-5mm}\\
recursively defines the extension. It is unique because of the unique factorisation\lb of integers into prime powers.
\end{proof}

\begin{rema}
There may be permutations $\xi$ of $h_1(\mD_n)$ that are not feasible for the equation $h_2(y)=\xi(h_1(y))$ for all $y\in\mD_n$ because of the multiplicative property of $h_2$.\\
\end{rema}

\begin{prop} \label{permutation2}
Let $n\in\mN>1$, $x,y\in\mD_n$, and $h_1,h_2\in\msI$ with $h_2(x)=\nu(h_1(x))$ where $\nu$ is an extended permutation of $h_1(\mD_n)$ according to Lemma~\refs{permutation1}. Then
\vspace*{-3mm}
\[\mZ\qs\sim_{\preceq_{d^{\around{n}}_{(h_1,+)}}} = \mZ\qs\sim_{\preceq_{d^{\around{n}}_{(h_2,+)}}}.\]
\vspace*{-8mm}\\
\end{prop}

\begin{proof}
As a consequence of Lemma~\refs{crit_QO}, we get for every $h\in\msI$
\vspace*{-3mm}
\[x \sim_{\preceq_{d^{\around{n}}_{(h,+)}}} y \iff d^{\around{n}}_{(h,+)}(1,x) = d^{\around{n}}_{(h,+)}(1,y) \iff 1-h(x) = 1-h(y).\]
\vspace*{-6mm}\\
Now,
\vspace*{-4mm}
\begin{align*}
x \sim_{\preceq_{d^{\around{n}}_{(h_1,+)}}} y &\iff h_1(x)=h_1(y) \iff \nu\bigl(h_1(x)\bigr)=\nu\bigl(h_1(y)\bigr)\\[-0.5ex]
&\iff h_2(x)=h_2(y) \iff x\preceq_{d^{\around{n}}_{(h_2,+)}} y.
\end{align*}
\vspace*{-14mm}\\
\end{proof}
\ \\\\

For the further investigation of the quotient sets $\mD_n\qs\!\sim_{\preceq_{d^{\around{n}}_{(h,+)}}}$, we distinguish\lb between $h\in\msI^{\around{n}}_{>0}$ and $h\in\msI^{\around{n}}_{=0}$ because of Proposition~\refs{zero}. We start with positive functions.

The finest partition of $\mD_n$ or $\mZ$ is achieved when all $h(x)$ are different on $\mD_n$. Then, each element of $\mD_n$ forms its own equivalence class, as is the case with $\mZ\qs\sim_{\preceq_{d^{\around{n}}_{(h,+)}}}$.

\begin{lemm} \label{finest_partition}
Let $n\in\mN>1$, $x,y\in\mD_n$, and $h\in\msI^{\around{n}}_{>0}$ such that $h(x)=h(y) \ \Longrightarrow \ x=y$. Then
\vspace*{-3mm}
\[\mD_n \ \simeq \ \mD_n\qs\sim_{\preceq_{d^{\around{n}}_{(h,+)}}}.\]
\vspace*{-9mm}\\
\end{lemm}

\begin{proof}
Again, by Lemma~\refs{crit_QO}, we get
\vspace*{-3mm}
\[x \sim_{\preceq_{d^{\around{n}}_{(h,+)}}} y \iff d^{\around{n}}_{(h,+)}(1,x) = d^{\around{n}}_{(h,+)}(1,y) \iff h(x) = h(y).\]
\vspace*{-14mm}\\
\end{proof}

\vspace*{3mm}
\begin{rema}
For example, the function $h(x)=1/x$ is admissible, multiplicative, and meets the requirements of the lemma.

However, the existence of functions $h\in\msI^{\around{n}}_{>0}$ with pairwise different function values\lb is not obvious, even if the function values of $h$ on $\mP^*_n$ are pairwise different. Let\lb $x,y\in\mP^*_n$ and $x\perp y$. Then, there could still be $z\in\mP^*_n$ with $h(z)=h(x\cdot y)=h(x)\cdot h(y)$ if $n$ has at least two distinct prime factors and at least three distinct prime power\lb factors, e.g. $n=12$, $h(2)=0.5$, $h(3)=0.6$, $h(4)=0.3$, and also $h(6)=0.3$.

Furthermore, we point out that a function $h\in\msI^{\around{n}}_{=0}$ can only have pairwise different values if $n$ is a prime power. It must vanish at a certain $y\in\mP^*_n$. If there were more than one prime factor of $n$, then all multiples of $y$ would have the same value of $h$, which is zero.
\end{rema}

Every occurrence of the same function values of $h\in\msI$ on $\mD_n$ leads to a coarser partition of $\mD_n$ and thus of $\mZ$. If $h(x)=h(y)$ then the corresponding equivalence classes coincide. The finest imaginable partition is not possible in this case. There are a lot of feasible combinations of equal function values.

On the other hand, in the coarsest non-trivial partition, $\mD_n\qs\sim_{\preceq_{d^{\around{n}}_{(h,+)}}}$ consists of only\vspace{-2mm} two equivalence classes. This can be the case when $h\in\msI^{\around{n}}_{=0}$.

\begin{lemm} \label{coarsest_partition}
Let $n\in\mN>1$.\\
There exists $h\in\msI^{\around{n}}_{=0}$ such that $h(x)=h(y)$ for all $x,y\in\mD_n>1$. 
\end{lemm}

\begin{proof}
We set $h(x)=0$ for all $x\in\mP^*_n>1$. Then also $h(x\cdot y)=h(x)\cdot h(y)=0$\lb for $x\perp y$. 
\end{proof}

\begin{rema}
There is no equivalent additive function $f\in\msA$ or equivalent multiplicative function $g\in\msM$ since $h(x)=0$ would have to be mapped to infinity.
\end{rema}
\ \\

We complete our investigations of the quotient sets $\mD_n\qs\sim_{\preceq_{d^{\around{n}}_{(h,+)}}}$ by discussing some special cases. The coarsest conceivable partition of $\mD_n$ or $\mZ$ is the trivial partition, in which all elements fall into the same, single equivalence class. It can only be generated by the function $h(x)=\mathit{1}(x)=1$ for all $x\in\mD_n$, cf. Lemma~\refs{coarsest}. Vice versa, the finest partition is achieved if all $h(x)$ are different on $\mD_n$, cf. Lemma~\refs{finest_partition}.

In all non-trivial cases, there are at least two equivalence classes, such as if $n$ is a prime number. For prime powers $n=p^k$, $p\in\mP$, $k\in\mN$, there can be two to $k+1$ equivalence classes depending on the values $h(p^i)$, $i=1,\dots,k$.\\

In Theorem~\refs{coro_metric2}, we proved that every $h\in\msI_1$ generally induces a metric on $\mD_n$. In this case, $1$ always forms its own equivalence class. That is, the class of all integers relatively prime to $n$ cannot coincide with another class containing an integer that has a common divisor with $n$.

\begin{lemm} \label{coprime_extra}
Let $n\in\mN>1$, $h\in\msI_1$, and $x\in\mD_n>1$. Then
\vspace*{-3mm}
\[ [x]_{\sim_{\preceq_{d^{\around{n}}_{h,+)}}}} \cap\  [1]_{\sim_{\preceq_{d^{\around{n}}_{(h,+)}}}} =\ \emptyset . \]
\vspace*{-14mm}\\
\end{lemm}
\begin{proof}
Suppose there exists $y\in [x]_{\sim_{\preceq_{d^{\around{n}}_{(h,+)}}}} \cap\  [1]_{\sim_{\preceq_{d^{\around{n}}_{(h,+)}}}}$. Then,

\[ y \sim_{\preceq_{d^{\around{n}}_{(h,+)}}} x \ \land \ y \sim_{\preceq_{d^{\around{n}}_{(h,+)}}} 1 \quad\iff\quad h(y)=h(x)=h(1)=1\ , \]
in contradiction to $h(x)<1$.
\end{proof}

If we further restrict the function $h$ to be positive, all function values for $x>1$ cannot be equal if $n$ has at least two different prime factors. Therefore, we can conclude that every function $h\in\msI_{01}$, i.e. $0<h(x)<1$ for $x>1$, leads to at least three equivalence classes if $n$ has more than one prime divisor.

\begin{lemm} \label{not_all_equal}
Let $h\in\msI_{01}$ and $n\in\mN$ have at least two prime divisors. Then there exist $x,y\in\mD_n>1$ such that $h(x)\ne h(y)$.
\end{lemm}
\begin{proof}
By assumption, there must exist $x,z\in\mP^*_n>1$ such that $x\perp z$. Then $x$ and $y=x\cdot z$ meet the requirements because $h(y)=h(x\cdot z)=h(x)\cdot h(z)<h(x)$.
\end{proof}
\begin{coro} \label{max_three}
Let $h\in\msI_{01}$, and $n\in\mN$ has at least two prime divisors.  Then, $\mD_n\qs\sim_{\preceq_{d^{\around{n}}_{(h,+)}}}$ has at least three elements.
\end{coro}
\begin{proof}
Cf. Lemmata~\ref{not_all_equal}~and~\refs{coprime_extra}.
\end{proof}
\ \\

Another approach is to characterise special subsets of integers that can be shown\lb to have different function values. We demonstrate the relationship between coprime and squarefree numbers.

\begin{lemm} \label{coprime_squarefree}
Let $h\in\msI$, and $h(x)\ne h(y)$ holds for all coprime $x,y\in\mN$ with $x\ne y$.\\
Then $h(s)\ne h(t)$ also holds for all squarefree $s,t\in\mN$ with $s\ne t$. 
\end{lemm}
\begin{proof}
If $h(x)\ne h(y)$ applies to all coprime $x,y$, then also to all coprime and squarefree $x,y$. Suppose there are squarefree $s,t\in\mN$ with $s\ne t$ and $h(s)=h(t)$. Then all prime divisors of $s$ and $t$ are single. Thus, $\frac{s}{\gcd(s,t)}$, $\gcd(s,t)$, and $\frac{t}{\gcd(s,t)}$ are squarefree and pairwise coprime. We get \vspace*{-3mm}
\[h(s)=h\left(\frac{s}{\gcd(s,t)}\cdot\gcd(s,t)\right)=h\left(\frac{s}{\gcd(s,t)}\right)\cdot h\bigl(\gcd(s,t)\bigr).\]\\[-2.5ex]
An analogous equation applies to $t$. The common factor can be reduced to \vspace*{-3mm}
\[h\left(\frac{s}{\gcd(s,t)}\right)=h\left(\frac{t}{\gcd(s,t)}\right),\]\\[-2.5ex]
which contradicts the assumption.
\end{proof}
\begin{rema}
In other words: If there are equal function values for different squarefree integers, then also for different coprime numbers.

The assertion applies accordingly to all other admissible functions $f\in\msA$ and \lb $g\in\msM$.
\end{rema}
\ \\

We derived some specific types of preorders on $\mZ$ that are induced by admissible arithmetic functions. In Proposition~\refs{permutation2}, we showed that the corresponding quotient sets are invariant against permutation of the values of the function. The function\lb values of $h\in\msI_1$ on $\mP^*_n$ can be chosen arbitrarily within the defined range. This can lead to many equivalence classes. We discussed several special cases.

The question remains how many such types of preorders there are for a given $n$.\lb It is the same question as how many partitions $\mD_n\qs\sim_{\preceq_{d^{\around{n}}}}$ or $\mZ\qs\sim_{\preceq_{d^{\around{n}}}}$ exist.

There are a lot of combinatorial possibilities. In the concluding remarks below,\lb we will bound their number by using some arithmetic functions, which we will\lb discuss in the next section.

\ \\\\


\section{Examples\phj} \label{examples}

Many arithmetic functions are commonly known and widely studied. For the\lb general case, we refer to textbooks \cite{Hildebrand_2013, Hardy_Wright_1975, McCarthy_1986, Schwarz_Spilker_1994, Cira_Smarandache_2016}. Unless otherwise specified, we set $x,k,n\in\mN$, $p\in\mP$. \enquote{$\log$} denotes the natural logaritm. We use the standard prime factorisation of $x\in\mN$
\[x=\prod_{i=1}^\infty p_i^{\alpha_i}, \quad p_i\in\mP,\alpha_i\in\mN_0.\]
The squarefree kernel or radical of $n$ is the product of distinct prime divisors of $x$.
\[\rad(x)=\prod_{p_i\mid x} p_i=\prod_{\alpha_i\ge1} p_i.\]

Several arithmetic functions are not admissible in the sense of Definition~\refs{admissible}.\lb However, some of those functions can easily be modified to result in admissible\lb functions, as we will demonstrate below. 
For other established arithmetic functions, no admissible functions derived from them are known. We want to point out these five examples that are neither additive nor multiplicative.

\exhead{Prime counting function}
\begin{align*} \tabalign
\pi(x) =\ &\sum_{p\le x} 1.
\end{align*}
\exhead{Mangoldt function}
\begin{align*} \tabalign
\Lambda(x) =\ &\begin{cases}
	\log(p) & \text{if \ } \exists \ p,k\ :\ x=p^k,\\ 0 & \text{otherwise.} \end{cases}
\end{align*}
\exhead{Chebyshev functions}
\begin{align*} \tabalign
\theta(x) =\ &\sum_{p\le x} \log(p). \\
\psi(x) =\ &\sum_{k=1}^{\infty}\sum_{p^k\le x} \log(p) = \sum_{n\le x} \Lambda(n).
\end{align*}
\exhead{Partition function}
\begin{align*} \tabalign
p(x) =\ &\Bigl| \bigl\{ (a_1,\dots,a_x) : 0<a_1\le\dots\le a_x \land a_1+\dots+a_x=x \bigr\} \Bigr|. \\
&\text{The function represents the number of possible partitions} \\
&\tb\text{of a natural number } x.
\end{align*}

Furthermore, for example, no non-trivial admissible function derived from the\lb totally multiplicative Liouville function is known. It distinguishes between natural numbers with an odd or even number of prime divisors (with repetition), see also $\Omega(x)$.

\exhead{Liouville function}
\begin{align*} \tabalign
\lambda(x) =\ &(-1)^{\sum_{i=1}^{\infty} \alpha_i}.\\
\end{align*}

We know that additive arithmetic functions together with pointwise addition form\lb a group with the identity element $\mathit{0}(x)\equiv0$ whereas multiplicative arithmetic\lb functions $g$ where $g(x)\ne0$ for all $x\in\mN$ together with pointwise multiplication form a group with the identity element $\mathit{1}(x)\equiv1$. $\bigl(\msA_0,+\bigr)$, $\bigl(\msM_0,\,\cdot\,\bigr)$, and $\bigl(\msI_1,\,\cdot\,\bigr)$ are subsemigroups of them, respectively.

The easiest way to modify such functions is to use operations inside these\lb semigroups. Pointwise sums or products can thus result in admissible functions.\lb Inverse elements are also candidates. Products of general multiplicative functions are suitable as well.

Below, we present examples of admissible functions related to well-established\lb additive or multiplicative arithmetic functions and summarise some relevant\lb properties. For functions belonging to $\msA_0$, $\msM_0$, or $\msI_1$, we additionally indicate the\lb corresponding functions of the triple characterised in Corollary~\refs{triple}.\\\\

The modification of an additive or multiplicative function described above always leads to another additive or multiplicative function, respectively. However, the\lb modification principles can also be applied to general functions that are neither\lb additive nor multiplicative and can nevertheless lead to admissible functions. We\lb begin with one such extraordinary example.

The following arithmetic functions have been defined using functional equations corresponding to those that apply to derivatives of real or complex functions. The arithmetic derivative \cite{Barbeau_1961, Ufnarovski_Ahlander_2003, Lava_Balzarotti_2013} is recursively defined by $D(1)=0$, $D(p)=1$\lb for $p\in\mP$, and $D(x\cdot y)=x\cdot D(y)+ D(x)\cdot y$ for $x,y\in\mN$. The respective logarithmic derivative \cite{Ufnarovski_Ahlander_2003, Lava_Balzarotti_2013} reads $\ld(1)=0$, $\ld(p)=1/p$ for $p\in\mP$, and $\ld(x\cdot y)=\ld(x)+ \ld(y)$\lb for $x,y\in\mN$ in analogy to the derivation of the logarithmic function. We get explicit representations when we use the standard prime factorisation.

The arithmetic derivative is neither additive nor multiplicative while the logarithmic derivative is an admissible and totally additive arithmetic function. It turns out to be the pointwise product of $D(x)$ and the admissible totally multiplicative function $1/x$ which results in an admissible totally additive function - a curious coincidence. Furthermore, we prove a special property of this function.

\exhead{Arithmetic derivative \emph{\cite{Barbeau_1961, Ufnarovski_Ahlander_2003, Lava_Balzarotti_2013}}}
\begin{align*} \tabalign
D(x) =\ &x\cdot\sum_{i=1}^{\infty} \frac{\alpha_i}{p_i}.
\end{align*}
\exhead{Logarithmic derivative \emph{\cite{Ufnarovski_Ahlander_2003, Lava_Balzarotti_2013}}}
\begin{align*} \tabalign
\ld(x) =\ &\frac{D(x)}{x} = \sum_{i=1}^{\infty} \frac{\alpha_i}{p_i}. \\
\ld(x) \in\ &\msA, \text{ totally additive.} \\
&\text{Equal function values occur for coprime numbers, e.g.: } \ld(4)=\ld(27)=1. \\
&\text{The function values for squarefree numbers are pairwise distinct}. \\
\msM\ni\  &\e^{\ld(x)}=\e^{\sum_{i=1}^{\infty} \frac{\alpha_i}{p_i}}=\prod_{i=1}^{\infty} \e^{\frac{\alpha_i}{p_i}}, \text{ totally multiplicative triple function}.\\
\msI\ni\  &\e^{-\ld(x)}=\e^{-\sum_{i=1}^{\infty} \frac{\alpha_i}{p_i}}=\prod_{i=1}^{\infty} \e^{-\frac{\alpha_i}{p_i}}, \text{ totally multiplicative triple function}.
\end{align*}

\begin{lemm}
All function values $\ld(x)$ for squarefree $x\in\mN$ are pairwise distinct.
\end{lemm}
\begin{proof}
Let $z\in\mN$ be squarefree. Then $\alpha_i\le1$ for all $i\in\mN$ and $\ld(z) = \sum_{p|z} \frac{1}{p}$.

We first prove that for all squarefree $z\in\mN>1$, $\ld(z)$ can be written as an\lb irreducible fraction. Its denominator is the product of the respective prime numbers dividing $z$. If $z\in\mP$ then $\ld(z)=\frac{1}{z}$ satisfies the claim. Now let $z$ have more than one prime divisor and $q\in\mP$ with $q|z$. By assumption, we get 

\[\ld(z) = \sum_{p|z} \frac{1}{p}=\frac{{\displaystyle\sum_{p|z}} \frac{z}{p}}{z}=
\frac{{\displaystyle\sum_{p|z \atop p\ne q}} \frac{z}{p}+\frac{z}{q}}{{\displaystyle\prod_{p|z}}p}=
\frac{q\cdot{\displaystyle\sum_{p|z\atop p\ne q}}\frac{z}{p\cdot q}+{\displaystyle\prod_{p|z\atop p\ne q}}p} {q\cdot{\displaystyle\prod_{p|z\atop p\ne q}}p}\ .\]\\[-0.5ex]
All but one of the summands of the numerator contain $q$. The other addend is a\lb product of other prime numbers and is not divisible by q. So, the fractions cannot be reduced.

Given $x,y\in\mN>1 $ with $x\ne y$. Then, $x$ and $y$ have distinct prime factors. The irreducible fractions $\ld(x)$ and $\ld(y)$ must be different. Furthermore, $\ld(1)=0$ and $\ld(z)>0$ for all $z>1$.
\end{proof}
\ \\[-1ex]

There are three simple functions which form a triple according to Corollary~\ref{triple} themselves. They result in the finest possible partition, described in Lemma~\refs{finest_partition}.

\exhead{Logarithmic function}
\begin{align*} \tabalign
\log(x)=\ &f(x) \in\msA, \text{ totally additive triple function.} \\
&\text{All function values are pairwise distinct}.
\end{align*}\vspace{-4mm}
\exhead{Identical function}
\begin{align*} \tabalign
x=\ &g(x) \in\msM, \text{ totally multiplicative triple function.} \\
&\text{All function values are pairwise distinct}.
\end{align*}\vspace{-4mm}
\exhead{Reciprocal function}
\begin{align*} \tabalign
\frac{1}{x}=\ &h(x) \in\msI, \text{ totally multiplicative triple function.} \\
&\text{All function values are pairwise distinct}. \\
\end{align*}

The following collection includes established additive functions related to\lb counts and sums of prime divisors. These functions are themselves admissible. No modification is needed.

Authors use different notations for the summation functions. We introduce\lb the Greek uppercase and lowercase letters upsilon following the given notation\lb of counting functions.

In case of squarefree $n$, both functions of each pair coincide for $x\in\mN$. Then\lb $\Omega(x)=\omega(x)$ and $\Upsilon(x)=\upsilon(x)$.

\exhead{Number of prime divisors (with repetition)}
\begin{align*} \tabalign
\Omega(x) =\ &\Bigl|\bigl\{z=p^k : p\in\mP \land k\in\mN \land z\mdiv x\bigr\}\Bigr|= \sum_{i=1}^{\infty} \alpha_i. \\
&\text{The function represents the number of different prime powers dividing } n.\\
\Omega(x) \in\ &\msA, \text{ totally additive}. \\
&\text{Equal function values occur for squarefree and thus also for coprime} \\
&\tb\text{numbers, e.g.: } \Omega(14)=\Omega(15)=2. 
\end{align*}
\vspace{-18mm}
\begin{align*} \tabalign
\msM \ni\ &\e^{\Omega(x)}=\e^{\sum_{i=1}^{\infty} \alpha_i}=\prod_{i=1}^{\infty} \e^{\alpha_i}, \text{ totally multiplicative triple function}. \\
\msI \ni\ &\e^{-\Omega(x)}=\e^{-\sum_{i=1}^{\infty} \alpha_i}=\prod_{i=1}^{\infty} \e^{-\alpha_i}, \text{ totally multiplicative triple function}.
\end{align*}
\exhead{Number of distinct prime divisors}
\begin{align*} \tabalign
\omega(x) =\ &\Bigl|\bigl\{p\in\mP : p\mdiv x\bigr\}\Bigr|= \sum_{p | x} 1. \\
\omega(x) \in\ &\msA, \text{ additive}. \\
&\text{Equal function values occur for squarefree and thus also for coprime} \\
&\tb\text{numbers, e.g.: } \omega(3)=\omega(5)=1. \\
\msM \ni\ &\e^{\omega(x)}=\e^{\sum_{p | x} 1}, \text{ multiplicative triple function}. \\
\msI \ni\ &\e^{-\omega(x)}=\e^{-\sum_{p | x} 1}, \text{ multiplicative triple function}.
\end{align*}
\exhead{Sum of prime divisors (with repetition) \emph{\cite{Chawla_1968, Lal_1969, Alladi_Erdoes_1977, Jakimczuk_2012}}}
\begin{align*} \tabalign
\Upsilon(x) =\ &\sum_{i=1}^{\infty} \alpha_i\cdot p_i. \\
\Upsilon(x) \in\ &\msA, \text{ totally additive}. \\
&\text{Equal function values occur for squarefree and thus also for coprime} \\
&\tb\text{numbers, e.g.: } \Upsilon(7)=\Upsilon(10)=7. \\
\msM\ni\ &\e^{\Upsilon(x)}=\e^{\sum_{i=1}^{\infty} \alpha_i\cdot p_i}=\prod_{i=1}^{\infty} \e^{\alpha_i\cdot p_i}, \text{ totally multiplicative triple function}. \\
\msI\ni\ &\e^{-\Upsilon(x)}=\e^{-\sum_{i=1}^{\infty} \alpha_i\cdot p_i}=\prod_{i=1}^{\infty} \e^{-\alpha_i\cdot p_i}, \text{ totally multiplicative triple function}.
\end{align*}
\exhead{Sum of distinct prime divisors \emph{\cite{Kerawala_1969, Alladi_Erdoes_1977}}}
\begin{align*} \tabalign
\upsilon(x) =\ &\sum_{p | x} p. \\
\upsilon(x) \in\ &\msA, \text{ additive}. \\
&\text{Equal function values occur for squarefree and thus also for coprime} \\
&\tb\text{numbers, e.g.: } \upsilon(5)=\upsilon(6)=5. \\
\msM \ni\ &\e^{\upsilon(x)}=\e^{\sum_{p | x} p}=\prod_{p | x} \e^{p}, \text{ multiplicative triple function}. \\
\msI \ni\ &\e^{-\upsilon(x)}=\e^{-\sum_{p | x} p}=\prod_{p | x} \e^{-p}, \text{ multiplicative triple function}. \\
\end{align*}

We continue with multiplicative functions which count or sum divisors or\lb coprimes. All these functions are members of $\msM_0$. In addition, we demonstrate some modifications. Especially, each of them can be normalised to a function of $\msI_1$ with interesting properties.

\exhead{Number of divisors}
\begin{align*} \tabalign
\nd(x) =\ &\Bigl|\bigl\{d\in\mN : d\mdiv x\bigr\}\Bigr|=\sum_{d\mid x} 1= \prod_{i=1}^{\infty} (1+\alpha_i). \\
&\nd(x) = 2^{\omega(x)} \text{ for squarefree } x. \\
\nd(x) \in\ &\msM, \text{ multiplicative}. \\
&\text{Equal function values occur for squarefree and thus also for coprime} \\
&\tb\text{numbers, e.g.: } \nd(14)=\nd(15)=4. \\
\msA \ni\ &\log(\nd(x))=\sum_{i=1}^{\infty} \log(1+\alpha_i), \text{ additive triple function}. \\
\msI \ni\ &\frac{1}{\nd(x)}=\prod_{i=1}^{\infty} \frac{1}{1+\alpha_i}, \text{ multiplicative triple function}.
\end{align*}
\exhead{Normalised number of divisors, variant 1}
\begin{align*} \tabalign
\frac{\nd(x)}{x} =\ &\frac{1}{x}\cdot\prod_{i=1}^{\infty} (1+\alpha_i) \in\msI_1, \text{ multiplicative}. \\
&\text{Equal function values occur, e.g.: } \frac{\nd(1)}{1}=\frac{\nd(2)}{2}=1 \text{ and } \frac{\nd(8)}{8}=\frac{\nd(12)}{12}=\frac{1}{2}. \\
&\text{The function values for coprime numbers } x,y>1 \text{ are pairwise distinct} \\
&\tb\text{and so for squarefree numbers } x,y>1 \text{ according to Lemma~\refs{coprime_squarefree}}. \\
\msA_0 \ni\ &-\log(\frac{\nd(x)}{x})=\log(x)-\log\bigl(\nd(x)\bigr)=\log(x)-\sum_{i=1}^{\infty} \log(1+\alpha_i), \\
&\text{ additive triple function}. \\
\msM_0 \ni\ &\frac{x}{\nd(x)}=x\cdot\prod_{i=1}^{\infty} \frac{1}{1+\alpha_i}, \text{ multiplicative triple function}.
\end{align*}
\begin{lemm}
The function values ${\displaystyle\frac{\nd(x)}{x}}$ and ${\displaystyle\frac{\nd(y)}{y}}$ for coprime $x,y\in\mN>1$ with $x\ne y$\lb are distinct.
\end{lemm}
\begin{proof}
We have ${\displaystyle\frac{\nd(2)}{2}=1}$, and ${\displaystyle\frac{\nd(z)}{z}=\prod_{i=1}^{\infty} \frac{1+\alpha_i}{p_i^{\alpha_i}}<1}$ for all $z\in\mN>2$. \\

The assertion holds if $x=2$ or $y=2$. Let now $x,y\in\mN>2$ be distinct and coprime.\\\\\\ Then, $x$ and $y$ have no common prime divisor. The denominators of the irreducibly\lb reduced fractions ${\displaystyle\frac{\nd(x)}{x}}$ and ${\displaystyle\frac{\nd(y)}{y}}$ must also have different prime divisors. Thus,\lb both terms are different.
\end{proof}
\exhead{Normalised number of divisors, variant 2}
\begin{align*} \tabalign
\frac{\nd(x)}{x^2} =\ &\frac{1}{x^2}\cdot\prod_{i=1}^{\infty} (1+\alpha_i) \in\msI, \text{ multiplicative}. \\
&\text{Equal function values occur,} \\
&\tb\text{e.g.: } \frac{\sigma(30000)}{30000^2}=\frac{4+1}{2^8}\cdot\frac{1+1}{3^2}\cdot\frac{4+1}{5^8}=\frac{1}{2^7\cdot3^2\cdot5^6} \\
&\tb\text{and } \frac{\sigma(36000)}{36000^2}=\frac{5+1}{2^{10}}\cdot\frac{2+1}{3^4}\cdot\frac{3+1}{5^6}=\frac{1}{2^7\cdot3^2\cdot5^6}\,. \\
&\text{The function values for coprime numbers are pairwise distinct} \\
&\tb\text{and so for squarefree numbers according to Lemma~\refs{coprime_squarefree}}. \\
\msA \ni\ &-\log(\frac{\nd(x)}{x^2})=2\cdot\log(x)-\log\bigl(\nd(x)\bigr), \text{ additive triple function}.\\
\msM \ni\ &\frac{x^2}{\nd(x)}=x^2\cdot\prod_{i=1}^{\infty} \frac{1}{1+\alpha_i}, \text{ multiplicative triple function}.
\end{align*}
\begin{lemm}
The function values ${\displaystyle\frac{\nd(x)}{x^2}}$ and ${\displaystyle\frac{\nd(y)}{y^2}}$ for coprime $x,y\in\mN$ with $x\ne y$\lb are distinct.
\end{lemm}
\begin{proof}
We have ${\displaystyle\frac{\nd(1)}{1^2}=1}$, and ${\displaystyle\frac{\nd(z)}{z^2}=\prod_{i=1}^{\infty} \frac{1+\alpha_i}{p_i^{2\cdot\alpha_i}}<1}$ for all $z\in\mN>1$. \\[-0.5ex]

The assertion holds if $x=1$ or $y=1$. Let now $x,y\in\mN>1$ be distinct and coprime.\lb Then, $x$ and $y$ have no common prime divisor. The denominators of the irreducibly\lb reduced fractions ${\displaystyle\frac{\nd(x)}{x^2}}$ and ${\displaystyle\frac{\nd(y)}{y^2}}$ must also have different prime divisors. Thus,\lb both terms are different.
\end{proof}

\exhead{Sum of divisors}
\begin{align*} \tabalign
\sigma(x) =\ &\sum_{d | x} d = \prod_{i=1}^{\infty} \frac{p_i^{\alpha_i+1}-1}{p_i-1}\,. \\
\sigma(x) \in\ &\msM, \text{ multiplicative}. \\
&\text{Equal function values occur for squarefree and thus also for coprime} \\
&\tb\text{numbers, e.g.: } \sigma(6)=\sigma(11)=12. \\
\end{align*}
\vspace{-15mm}
\begin{align*} \tabalign
\msA \ni\ &\log(\sigma(x)), \text{ additive triple function}.\\
\msI \ni\ &\frac{1}{\sigma(x)} = \prod_{i=1}^{\infty} \frac{p_i-1}{p_i^{\alpha_i+1}-1}, \text{ multiplicative triple function}.
\end{align*}
\exhead{Modified sum of divisors}
\begin{align*} \tabalign
\frac{\sigma(x)}{x} =\ &\frac{1}{x}\cdot\prod_{i=1}^{\infty} \frac{p_i^{\alpha_i+1}-1}{p_i-1} \in\msM, \text{ multiplicative}. \\
&\text{Sum of reciprocal divisors :   } \frac{\sigma(x)}{x} = \frac{1}{x}\cdot\sum_{d | x} d = \frac{1}{x}\cdot\sum_{d | x} \frac{x}{d} = \sum_{d | x} \frac{1}{d}\,. \\
&\text{Equal function values occur, e.g.: } \frac{\sigma(6)}{6}=\frac{\sigma(28)}{28}=2 \text{ (perfect numbers)}, \\
&\tb\text{and }\frac{\sigma(30)}{30}=\frac{\sigma(120)}{120}=2.4 \text{ (abundant numbers)}. \\
&\text{There are no known equal function values for different coprime numbers}. \\
&\tb\text{Pairwise distinct function values for coprime numbers are conjectured}. \\
&\text{The function values for coprime numbers are pairwise distinct} \\
&\tb\text{if there were no odd multiply perfect numbers}. \\
&\text{The function values for squarefree numbers are pairwise distinct}. \\
\msA \ni\ &\log(\frac{\sigma(x)}{x})=\log\bigl(\sigma(x)\bigr)-\log(x), \text{ additive triple function}.\\
\msI \ni\ &\frac{x}{\sigma(x)} = x\cdot\prod_{i=1}^{\infty} \frac{p_i-1}{p_i^{\alpha_i+1}-1}, \text{ multiplicative triple function}.
\end{align*}
\begin{lemm} \label{odd_coprime}
The function values ${\displaystyle\frac{\sigma(x)}{x}}$ and ${\displaystyle\frac{\sigma(y)}{y}}$ for coprime $x,y\in\mN$ with $x\ne y$\lb are distinct if there were no odd multiply perfect numbers.
\end{lemm}
\begin{proof}
For all $z\in\mN$, we get
\vspace*{-2mm}
\[\frac{\sigma(x)}{z} = \frac{1}{z}\cdot\prod_{i=1}^{\infty} \frac{p_i^{\alpha_i+1}-1}{p_i-1} = \prod_{i=1}^{\infty} \frac{p_i^{\alpha_i+1}-1}{p_i^{\alpha_i}\cdot(p_i-1)} = \prod_{i=1}^{\infty} \frac{\sum_{k=0}^{\alpha_i} p_i^k}{p_i^{\alpha_i}}\,.\] 
\vspace*{-2mm} \\
Furthermore, ${\displaystyle\frac{\sigma(1)}{1}}=1$, and ${\displaystyle\frac{\sigma(z)}{z}}>1$ for $z>1$.
\vspace*{-3mm} \\

A multiply perfect number $m\in\mN$ divides its sum of divisors. i.e. ${\displaystyle\frac{\sigma(m)}{m}}\in\mN$.\lb If there were no odd multiply perfect numbers, they would all be even and share\lb the common divisor $2$. Therefore, there would not be two different and coprime\lb $x,y\!\in\!\mN>1$ such that the corresponding function values ${\displaystyle\frac{\sigma(y)}{y}}$ and ${\displaystyle\frac{\sigma(y)}{y}}$ are integers.

\pagebreak
Let now $x,y\in\mN$ be distinct and coprime. Then, $x$ and $y$ have no common prime divisor. The assertion of the lemma holds if $x=1$ or $y=1$. Otherwise, for $x,y>1$,\lb the denominators of the irreducibly reduced fractions ${\displaystyle\frac{\sigma(x)}{x}}$ and ${\displaystyle\frac{\sigma(y)}{y}}$ have distinct prime divisors if one of them is not an integer. Thus, both terms must be different\lb if the assumption on odd multiply perfect numbers holds.
\end{proof}

\vspace*{-2mm}
\begin{lemm}
All function values ${\displaystyle\frac{\sigma(x)}{x}}$ for squarefree $x\in\mN$ are pairwise distinct.
\end{lemm}
\begin{proof}
Let $z\in\mN$ be squarefree. Then $\alpha_i\le1$ for all $i\in\mN$, and
\vspace*{-2mm}
\[\frac{\sigma(z)}{z} = \frac{1}{z}\cdot\prod_{p|z} (p+1) = \prod_{p|z}  \frac{p+1}{p}\ .\]
\vspace*{-4mm} \\
According to Lemma~\refs{coprime_squarefree}, it is sufficient to prove the assertion for coprime and\lb squarefree $x,y\in\mN$.

Let now $x,y\in\mN$ be squarefree and coprime. Then, all prime divisors of them are single factors, and $x$ and $y$ have no common prime divisor. W.l.o.g. let $q\in\mP>2$ be the largest prime divisor of $y$, and all prime divisors of $x$ are smaller than $q$. In the case of $q=2$, we get $x=1$ and $y=2$ with $1=\frac{\sigma(1)}{1}\ne\frac{\sigma(2)}{2}=\frac{3}{2}$.

All prime factors of the denominator of ${\displaystyle\frac{\sigma(x)}{x}=\prod_{p | x} \frac{p+1}{p}=\frac{\prod_{p |x} (p+1)}{\prod_{p | x} p}}$ are smaller than $q$. This also applies in the case of $x=1$. However, the denominator\lb of ${\displaystyle\frac{\sigma(y)}{y}=\prod_{p | y} \frac{p+1}{p}=\frac{\prod_{p | y} (p+1)}{\prod_{p | y} p}}$ contains $q$, and $q$ cannot be reduced because\lb all factors $p+1$ in the numerator are composite and their prime divisors are smaller than $q$. Thus, both terms must be different.
\end{proof}

\exhead{Normalised sum of divisors}
\begin{align*} \tabalign
\frac{\sigma(x)}{x^2} =\ &\frac{1}{x^2}\cdot\prod_{i=1}^{\infty} \frac{p_i^{\alpha_i+1}-1}{p_i-1} \in\msI, \text{ multiplicative}. \\
&\text{No equal function values are known}. \\
&\tb\text{Pairwise distinct function values are conjectured}. \\
&\text{The function values for coprime numbers are pairwise distinct} \\
&\tb\text{and so for squarefree numbers according to Lemma~\refs{coprime_squarefree}}. \\
\msA \ni\ &-\log\left(\frac{\sigma(x)}{x^2}\right)=2\cdot\log(x)-\log\bigl(\sigma(x)\bigr), \text{ additive triple function}.\\
\msM \ni\ &\frac{x^2}{\sigma(x)} = x^2\cdot\prod_{i=1}^{\infty} \frac{p_i-1}{p_i^{\alpha_i+1}-1}, \text{ multiplicative triple function}.
\end{align*}

\begin{lemm}
The function values ${\displaystyle\frac{\sigma(x)}{x^2}}$ and ${\displaystyle\frac{\sigma(y)}{y^2}}$ for coprime $x,y\in\mN$ with $x\ne y$\lb are distinct.
\end{lemm}
\begin{proof}
For all $z\in\mN$, we get
\vspace*{-3mm}
\[\frac{\sigma(x)}{z^2} = \frac{1}{z^2}\cdot\prod_{i=1}^{\infty} \frac{p_i^{\alpha_i+1}-1}{p_i-1} = \prod_{i=1}^{\infty} \frac{p_i^{\alpha_i+1}-1}{p_i^{2\cdot\alpha_i}\cdot(p_i-1)} = \prod_{i=1}^{\infty} \frac{\sum_{k=0}^{\alpha_i} p_i^k}{p_i^{2\cdot\alpha_i}}\,.\] 
\vspace*{-4mm} \\
Furthermore, ${\displaystyle\frac{\sigma(1)}{1^2}}=1$, and ${\displaystyle\frac{\sigma(z)}{z^2}}<1$ for $z>1$.
\vspace*{-2mm} \\

Let now $x,y\in\mN$ be distinct and coprime. Then, $x$ and $y$ have no common\lb prime divisor. The assertion holds if $x=1$ or $y=1$. Otherwise, for $x,y>1$, the denominators of the irreducibly reduced fractions ${\displaystyle\frac{\sigma(x)}{x^2}}$ and ${\displaystyle\frac{\sigma(y)}{y^2}}$ have distinct prime divisors. Thus, both terms must be different.
\end{proof}

\exhead{Euler's totient function}
\begin{align*} \tabalign
\varphi(x) =\ &\Bigl|\bigl\{k\in\mN\le x\ :\ k\perp x\bigr\}\Bigr|= x\cdot\prod_{p | x} \left(1-\frac{1}{p}\right). \\
&\text{Number of natural numbers } \le n \text{ that are relatively prime to  } n. \\
\varphi(x) \in\ &\msM_0, \text{ multiplicative}. \\
&\text{Equal function values occur for squarefree and thus also for coprime} \\
&\tb\text{numbers, e.g.: } \varphi(1)=\varphi(2)=1 \text{ and } \varphi(13)=\varphi(21)=12. \\
\msA_0 \ni\ &\log\bigl(\varphi(x)\bigr) = \log(x)+\sum_{p | x} \log(1-p)-\sum_{p | x} \log(p), \text{ additive triple function}.\\
\msI_1 \ni\ &\frac{1}{\varphi(x)} = \frac{1}{x}\mdot\prod_{p | x} \!\left(\frac{p}{p-1}\right)\! = \frac{1}{x}\mdot\prod_{p | x} \!\left(1+\frac{1}{p-1}\right)\!, \!\text{ multiplicative triple function}.
\end{align*}
\exhead{Normalised Euler's function, variant 1}
\begin{align*} \tabalign
\frac{\varphi(x)}{x}  =\ &\prod_{p | x} \left(1-\frac{1}{p}\right) = \frac{\varphi(\rad(x))}{\rad(x)}\ \in \msI, \text{  multiplicative}. \\
&\text{Equal function values occur, e.g.: } \frac{\varphi(6)}{6}=\frac{\varphi(12)}{12}=\frac{1}{3}. \\
&\text{The function values for coprime numbers are pairwise distinct} \\
&\tb\text{and so for squarefree numbers according to Lemma~\refs{coprime_squarefree}}. \\
\msA \ni\ &-\log\left(\frac{\varphi(x)}{x}\right)=\log\biggl(\prod_{p | x} \frac{p}{1-p}\biggr)=\sum_{p | x} \log\left(1+\frac{1}{p-1}\right), \\ 
&\text{ additive triple function}.  \\
\end{align*}
\vspace{-16mm}
\begin{align*} \tabalign
\msM \ni\ &\frac{x}{\varphi(x)} = \prod_{p | x} \left(\frac{p}{p-1}\right) =\prod_{p | x} \left(1+\frac{1}{p-1}\right), \text{ multiplicative triple function}.
\end{align*}
\begin{lemm}
The function values ${\displaystyle\frac{\varphi(x)}{x}}$ and ${\displaystyle\frac{\varphi(y)}{y}}$ for coprime $x,y\in\mN$ with $x\ne y$\lb are distinct.
\end{lemm}
\begin{proof}
For all $z\in\mN$, we have ${\displaystyle\frac{\varphi(z)}{z}  = \frac{\varphi\bigl(\rad(z)\bigr)}{\rad(z)}}$. Furthermore, ${\displaystyle\frac{\varphi(1)}{1}}=1$, and\lb ${\displaystyle\frac{\sigma(z)}{z}}<1$ for $z>1$. Therefore, it is sufficient to prove the assertion for coprime and squarefree $x,y\in\mN$.

Let now $x,y\in\mN$ be squarefree and coprime. Then, all prime divisors of them are single factors, and $x$ and $y$ have no common prime divisor. W.l.o.g. let $q\in\mP$ be the largest prime divisor of $y$, and all prime divisors of $x$ are smaller than $q$.\vspace*{2mm}

All prime factors of the denominator of ${\displaystyle\frac{\varphi(x)}{x}=\prod_{p | x} \left(1-\frac{1}{p}\right)=\frac{\prod_{p |x} (p-1)}{\prod_{p | x} p}}$ are\vspace*{1mm}\lb smaller than $q$. This also applies in the case of $x=1$. However, the denominator\vspace*{1mm} of ${\displaystyle\frac{\varphi(y)}{y}=\prod_{p | y} \left(1-\frac{1}{p}\right)=\frac{\prod_{p | y} (p-1)}{\prod_{p | y} p}}$ contains $q$, and $q$ cannot be reduced because\lb all factors $p-1$ of the numerator are smaller than $q$. Thus, both terms must be\lb different.
\end{proof}
\exhead{Normalised Euler's function, variant 2}
\begin{align*} \tabalign
\frac{\varphi(x)}{x^2}  =\ &\frac{1}{x}\cdot\prod_{p | x} \left(1-\frac{1}{p}\right) \in \msI, \text{  multiplicative}. \\
&\text{No equal function values are known}. \\
&\tb\text{Pairwise distinct function values are conjectured}. \\
&\text{The function values for coprime numbers are pairwise distinct} \\
&\tb\text{and so for squarefree numbers according to Lemma~\refs{coprime_squarefree}}. \\
\msA \ni\ &-\log\left(\frac{\varphi(x)}{x^2}\right)=\log\biggl(\prod_{p | x} \frac{p}{1-p}\biggr)-\log(x), \text{ additive triple function}. \\
\msM \ni\ &\frac{x^2}{\varphi(x)} = x\mdot\prod_{p | x} \!\left(\frac{p}{p-1}\right) \!=x\mdot\prod_{p | x} \!\left(1+\frac{1}{p-1}\right)\!, \text{ multiplicative triple function}.
\end{align*}
\begin{lemm}
The function values ${\displaystyle\frac{\varphi(x)}{x^2}}$ and ${\displaystyle\frac{\varphi(y)}{y^2}}$ for coprime $x,y\in\mN$ with $x\ne y$\lb are distinct.
\end{lemm}
\begin{proof}
For all $z\in\mN$, we get
\vspace*{-3mm}
\[\frac{\varphi(x)}{z^2} = \frac{1}{z}\cdot\prod_{p | z} \left(1-\frac{1}{p}\right) = \frac{1}{z}\cdot\prod_{i=1\atop p_i|z}^{\infty} \frac{p_i-1}{p_i} = \prod_{i=1\atop p_i|z}^{\infty} \frac{p_i-1}{p_i^{\alpha_i+1}}\,.\] 
\vspace*{-5mm} \\
Furthermore, ${\displaystyle\frac{\varphi(1)}{1^2}}=1$, and ${\displaystyle\frac{\varphi(z)}{z^2}}<1$ for $z>1$.
\vspace*{-2mm} \\

Let now $x,y\in\mN$ be distinct and coprime. Then, $x$ and $y$ have no common prime divisor. The assertion holds if $x=1$ or $y=1$. Otherwise, for $x,y>1$, the denominators of the irreducibly reduced fractions ${\displaystyle\frac{\varphi(x)}{x^2}}$ and ${\displaystyle\frac{\varphi(y)}{y^2}}$ have distinct prime divisors. Thus, both terms must be different.
\end{proof}

\exhead{Pillai's function \emph{\cite{Pillai_1933, Broughan_2001, Toth_2010}}}
\begin{align*} \tabalign
P(x) =\ &\sum_{k=1}^{x} \gcd(k,x) = \sum_{d|x} d\mdot\varphi(\frac{x}{d}) = x\mdot\prod_{i=1}^{\infty} \Biggl(1+\alpha_i\mdot\left(1-\frac{1}{p_i}\right)\Biggr). \\
P(x) \in\ &\msM, \text{ multiplicative}. \\
&\text{Equal function values occur for squarefree and thus also for coprime} \\
&\tb\text{numbers, e.g.: } P(15)=P(23)=45. \\
\msA \ni\ &\log(P(x)) = \log(x)+\sum_{i=1}^{\infty} \log\Biggl(1+\alpha_i\mdot\left(1-\frac{1}{p_i}\right)\Biggr), \text{ additive triple function}.\\
\msI \ni\ &\frac{1}{P(x)}, \text{ multiplicative triple function}.
\end{align*}
\exhead{Normalised Pillai's function}
\begin{align*} \tabalign
\frac{P(x)}{x\mdot\nd(x)} =\ &\prod_{i=1}^{\infty} \frac{\left(1-\frac{1}{p_i}\right)\mdot\alpha_i+1}{\alpha_i+1} = \prod_{i=1}^{\infty} \left(1-\frac{\alpha_i}{p_i\mdot(\alpha_i+1)}\right) \in \msI, \text{  multiplicative}. \\
&\text{Equal function values occur for squarefree and thus also for coprime} \\
&\tb\text{numbers, e.g.: } \frac{P(2)}{2\mdot\nd(2)}=\frac{P(15)}{15\mdot\nd(15)}=\frac{3}{4}. \\
\msA \ni\ &-\log\left(\frac{P(x)}{x\mdot\nd(x)}\right)=\sum_{i=1}^{\infty} \log\left(\frac{\alpha_i+1}{\left(1-\frac{1}{p_i}\right)\mdot\alpha_i+1}\right)= \\
&=\sum_{i=1}^{\infty} \log\left(1+\frac{\alpha_i}{p_i\mdot(\alpha_i+1)-\alpha_i}\right), \text{ additive triple function}. \\
\msM \ni\ &\frac{x\mdot\nd(x)}{P(x)} = \prod_{i=1}^{\infty} \frac{\alpha_i+1}{\left(1-\frac{1}{p_i}\right)\mdot\alpha_i+1} = \prod_{i=1}^{\infty} \left(1+\frac{\alpha_i}{p_i\mdot(\alpha_i+1)-\alpha_i}\right), \\
&\text{ multiplicative triple function}.
\end{align*}

Our final examples are admissible multiplicative functions bounded by $[0,1]$ that contain zero function values or can be modified or normalised to such. Those functions $h_0\in\msI_0\setminus\msI_1$ do not belong to any triple in the sense of Corollary~\refs{triple}.

\exhead{M\"obius function}
\begin{align*} \tabalign
\mu(x) =\ &\begin{cases}
	(-1)^{\omega(x)} & \text{for squarefree }x, \\
	0 & \text{else}, \end{cases} \text{\qquad multiplicative}. \\
&\text{The function indicates whether the number of prime divisors} \\
&\tb\text{of a squarefree natural number is odd or even}. \\
&\text{It is not admissible because e.g. } \mu(5)<0.
\end{align*}
\exhead{Modified M\"obius function}
\begin{align*} \tabalign
\mu^2(x) =\ &\begin{cases}
	(-1)^{2\cdot\omega(x)}=1 & \text{for squarefree }x, \\
	0 & \text{else}, \end{cases} \qquad	\in\ \msI_0\setminus\msI_1, \text{  multiplicative}. \\
&\text{The function indicates squarefree numbers}. \\
&\text{It leads to only two equivalence classes}.
\end{align*}
\exhead{Normalised M\"obius function}
\begin{align*} \tabalign
\frac{\mu^2(x)}{x} =\ &\frac{\bigl|\mu(x)\bigr|}{x} = \begin{cases}
	\frac{1}{x} & \text{for squarefree }x, \\
	0 & \text{else}, \end{cases} \qquad	\in\ \msI_0\setminus\msI_1, \text{  multiplicative}. \\
&\text{Equal function values occur for coprime numbers, e.g. } \frac{\mu^2(4)}{4}=\frac{\mu^2(9)}{9}=0. \\
&\text{The function values for squarefree numbers are pairwise distinct,} \\
&\tb\text{ as well as for } \frac{1}{x} .
\end{align*}

\exhead{Principal Dirichlet character of modulus $k\in\mN>1$}
\begin{align*} \tabalign
\chi_k(x) =\ &\begin{cases}
	1 & \text{for }\gcd(x,k)=1, \\
	0 & \text{for }\gcd(x,k)>1. \end{cases} \\
&\text{The function indicates the coprimeness to } k. \\
&\chi_k(x) = \chi_k\bigl(\rad(x)\bigr), \quad \chi_k(r\mdot k+1)=1,\ r\in\mN. \\
\chi_k(x) \in\ &\ \msI_0\setminus\msI_1, \text{ completely multiplicative}. \\
&\text{It leads to only two equivalence classes}.
\end{align*}
\exhead{Normalised principal Dirichlet character of modulus $k\in\mN>1$}
\begin{align*} \tabalign
\frac{\chi_k(x)}{x} =\ &\begin{cases}
	\frac{1}{x} & \text{for }\gcd(x,k)=1, \\
	0 & \text{for }\gcd(x,k)>1, \end{cases} \qquad \in\ \msI_0\setminus\msI_1, \text{  completely multiplicative}. \\
&\text{Equal function values occur for squarefree and thus also for coprime}\\
&\tb\text{ numbers, if $k$ has at least two distinct prime divisors $p$ and $q$.}\\
&\tb\text{Then e.g.: } \frac{\chi_k(p)}{p}=\frac{\chi_k(q)}{q}=0. \\
&\text{The function values for squarefree numbers  that are coprime to } k, \text{are} \\
&\tb\text{pairwise distinct, as well as for } \frac{1}{x} .
\end{align*}

\exhead{Unit function}
\begin{align*} \tabalign
\epsilon(x) =\ &\begin{cases}
	1 & \text{for }x=1, \\
	0 & \text{for }x>1, \end{cases} \text{\qquad completely multiplicative}. \\
&\text{The function represents the identity element of the Dirichlet convolution}. \\
&\text{It is the product of the Dirichlet characters: }\ \epsilon(x) = \prod_{k=2}^{\infty} \chi_k(x) = \prod_{p\in\mP} \chi_p(x). \\
\epsilon(x) \in\ &\ \msI_0\setminus\msI_1, \text{  completely multiplicative}. \\
&\text{The function values for } x>1 \text{ are all the same}. \\
&\text{It leads to only two equivalence classes}.
\end{align*}

The unit function $\epsilon(x)$ coincides with the special  principal Dirichlet character $\chi_n(x)$ on $\mD_n$ for all $n$. So, both functions lead to the same equivalence classes. According to Lemma~\refs{coarsest_partition}, we get $h(x)=\chi_n(x)\in\msI^{\around{n}}_{=0}$. By definition, these functions lead to one of the coarsest possible partition of $\mZ$ where $\mZ\qs\sim_{\preceq_{d^{\around{n}}_{(h,+)}}}$ consists of only two equivalence classes. The class $[1]$ includes all numbers coprime to $n$, and the class $[n]$ comprises all numbers sharing a common divisor with $n$. 

In the general case including $k\ne n$ however, the resulting partition also consists of only two equivalence classes because $1$ is equivalent to all numbers coprime to $k\cdot n$, and $n$ is equivalent to all numbers sharing a common divisor with $k\cdot n$. 

\pagebreak


\section*{Concluding remarks}
\addcontentsline{toc}{section}{\phj Concluding remarks\phj}
\stepcounter{section}

In this paper, we have demonstrated the construction of pseudometrics and preorders on the set of integral numbers given an admissible arithmetic function. Integers\lb can hereby be ordered in terms of their divisibility properties with respect to a given modulus $n\in\mN$. Various examples have been compiled in the previous section.

We raised the issue of what types of preorders there are for a specific $n$. Another question is which admissible functions cause the same preorder on $\mZ$ or the same quotient set $\mZ\qs\sim_{\preceq_{d^{\around{n}}}}$.\\

We have already proved some essential statements related to the characterisation\lb of the considered preorders.

According to Lemma~\ref{extended_pseudometric} in combination with Corollary~\refs{coro_metric1}, the function values of an arithmetic function for all $x\in\mD_n$ uniquely determine the induced pseudometric\lb or preorder on $\mZ$. Just knowing the function values for all $x\in\mP^*_n$ is sufficient\lb for admissible arithmetic functions that are either additive or multiplicative by\lb Definition~\refs{admissible}.

In Lemma~\ref{permutation1} and Proposition~\refs{permutation2}, we proved that the quotient sets concerning the induced preorders are invariant with respect to permutations of the function values\lb of the respective admissible functions for all $x\in\mP^*_n$.\\

In Section ~\refs{AFO}, we examined special constellations of admissible functions and\lb moduli $n$ that lead to a different number of equivalence classes. A trivial function where all function values are equal results in only one equivalence class, namely\lb the trivial partition. Apart from that, the coarsest possible partition contains only two different classes, $[1]$ and $[n]$, such as if $n$ is a prime number.  Another example constellation was considered in Lemma~\refs{coprime_extra}.

The finest possible partition can be achieved if all function values on $\mD_n$ are\lb different, as described in Lemma~\refs{finest_partition}. All possibilities are conceivable between the extreme variants. A special case was examined in Corollary~\refs{max_three}. If $n=p^k$ is a prime power, then all feasible partitions are possible. The function values can be chosen accordingly.

In general, a finer partition can distinguish between more items of the divisibility\lb properties with respect to the given $n$. A coarse partition with only two classes,\lb however, indicates a single specific property.
\ \\\\

In the following, we summarise example functions that have pairwise distinct\lb values on $\mD_n$ and lead to the finest possible partition, and those that contain equal function values on $\mD_n$. For this, we consider separately the cases of functions with and without equal values for squarefree and coprime $x$, whereby three cases remain possible according to Lemma~\refs{coprime_squarefree}. Functions that result in the coarsest partition\lb complete the overview of five groups of admissible arithmetic functions.\\[3mm]

\begin{enumerate}
\item The function has pairwise distinct function values. This leads to the finest\lb partition.\\
Such example functions are ${\displaystyle\log(x),\  x \text{, and } \frac{1}{x}}$.\\
The same is conjectured for ${\displaystyle\frac{\sigma(x)}{x^2},\ \frac{x^2}{\sigma(x)},\ \frac{\varphi(x)}{x^2}, \text{ and } \frac{x^2}{\varphi(x)}}$.

\item There exist equal function values in general. For coprime numbers, however,\lb the function values are pairwise distinct and so for squarefree numbers.\\
This applies to ${\displaystyle\frac{\nd(x)}{x^2},\ \frac{x^2}{\nd(x)},\ \frac{\varphi(x)}{x},\ \frac{x}{\varphi(x)} \text{, and } \frac{\mu^2(x)}{x}}$.\\
It is also conjectured for ${\displaystyle\frac{\sigma(x)}{x} \text{ and } \frac{x}{\sigma(x)}}$. \\
For $x>1$, the assertion holds for ${\displaystyle\frac{\nd(x)}{x} \text{ and } \frac{x}{\nd(x)}}$.

\item Equal function values also exist for coprime numbers. For squarefree numbers, however, the function values are pairwise distinct.\\
Examples are ${\displaystyle\ld(x) \text{ and } \frac{\mu^2(x)}{x}}$.

\item
Equal function values even occur for squarefree and thus for coprime numbers.\\
This is valid for 
${\displaystyle\Omega(x),\ \omega(x),\ \Upsilon(x),\ \upsilon(x),\ \nd(x),\ \frac{1}{\nd(x)},\ \sigma(x),\ \frac{1}{\sigma(x)},\ \varphi(x),\ \frac{1}{\varphi(x)}}$,
${\displaystyle P(x),\ \frac{1}{P(x)},\ \frac{P(x)}{x\cdot\nd(x)} \text{, and } \frac{x\cdot\nd(x)}{P(x)}}$.\\
 If $k$ has at least two distinct prime factors, ${\displaystyle\frac{\chi_k(x)}{x}}$ also belongs to this group.

\item
The function has only two different function values and leads to one of the\lb coarsest non-trivial partition.\\
This is the case with ${\displaystyle\mu^2(x)  \text{ and } \chi_k(x)}$.\\
$\epsilon(x)$ even has the same function values for all $x>1$.
\end{enumerate}
\ \\

We discussed a variety of admissible functions with different proportions of equal function values. The above gradation of functions applies to all natural numbers and thus to $D_n$ for all $n$. For a specific $D_n$, the derived partition could be even finer.

In Section~\refs{AFO}, we asked the question how many types of preorders on $\mD_n$ or $\mZ$ there are for a given $n$. An equivalent formulation for this is how many partitions $\mD_n / \sim_{\preceq_{d^{\around{n}}}}$ or $\mZ / \sim_{\preceq_{d^{\around{n}}}}$ exist. There are a lot of combinatorial possibilities. We will now bound their number using some of the arithmetic functions introduced in the last section.\\\\

Each pseudometric of $d^{\around{n}}$ according to Lemma~\ref{function_metric} and Definition~\ref{extended_function} is uniquely defined by the values of its generation function for all $x\in\mD_n$. The value $f(1)=0$ is fixed. The maximum number of possible partitions, apart from permutations by Proposition~\refs{permutation2}, is
\vspace*{-2mm}
\[p\bigl(\nd(n)-1\bigr),\]
\vspace*{-6mm}\\
where $p$ is the partition function, and $\nd$ represents the number of divisors.
\ \\

The case of pseudometrics induced by admissible arithmetic functions according to Definition~\ref{admissible} is more complicated. We emphasised that not all pseudometrics\lb on $\mD_n$ can be induced by such functions. They form strict subsets of $\msG$. The maximum\lb number of possible partitions derived from admissible arithmetic functions must\lb therefore be less than $p(\nd(n)-1)$ in general.

On the other hand, we know that a metric induced by an admissible function\lb is uniquely determined by the function values for all $x\in\mP^*_n$. The function values of $y\in\mD_n\setminus\mP^*_n$ cannot be chosen arbitrarily because of the additive or multiplicative properties of admissible functions. Thus, we can conclude that the maximum number of possible partitions, apart from permutation, is in this case
\vspace*{-3mm}
\[p\bigl(\Omega(n)-1\bigr).\]
\vspace*{-8mm}\\
The function $\Omega(n)$ returns the number of distinct primes powers dividing $n$.
\ \\\\

There are still some open questions. No equal function values are known for each\lb of the functions ${\displaystyle\frac{\sigma(x)}{x^2}}$ or ${\displaystyle\frac{\varphi(x)}{x^2}}$ for arbitrary different numbers. However, the\lb corresponding function values for different coprime numbers have been proved to be distinct, respectively.

Furthermore, no equal function values are known for the function ${\displaystyle\frac{\sigma(x)}{x}}$ for different coprime numbers. Though, it has been proved that all function values for squarefree $x\in\mN$ are pairwise distinct.

So, we pose the following three conjectures and put them up for discussion.

\begin{conj}
All function values of ${\displaystyle\frac{\sigma(x)}{x^2}}$ for $x\in\mN$ are pairwise distinct.
\end{conj}

\begin{conj}
All function values of ${\displaystyle\frac{\varphi(x)}{x^2}}$ for $x\in\mN$ are pairwise distinct.
\end{conj}

\begin{conj}  \label{distinct_coprime}
The function values ${\displaystyle\frac{\sigma(x)}{x}}$ and ${\displaystyle\frac{\sigma(y)}{y}}$ for coprime $x,y\in\mN$ with $x\ne y$ are distinct.
\end{conj}

These conjectures also apply to the reciprocal functions ${\displaystyle\frac{x^2}{\sigma(x)}}$, ${\displaystyle\frac{x^2}{\varphi(x)}}$, and ${\displaystyle\frac{x}{\sigma(x)}}$,\linebreak respectively, \ because $\bigl(\msM,\,\cdot\,\bigr)$ and $\bigl(\msI_0,\,\cdot\,\bigr)$ are isomorphic according to Proposition~\refs{iso-binary}. \\[-1mm]

The last conjecture is related to the well-known odd $k$-perfect number conjecture.\lb A natural number $n$ is called multiply perfect or $k$-perfect if $\sigma(n)=k\cdot n$ for\lb any $k\in\mN \ge2$. In the case $k=2$, $n$ is a perfect number. No odd $k$-perfect\lb numbers are known for any integer $k\ge2$ \cite{Hardy_Wright_1975, Guy_2004, Hildebrand_2013}. It is conjectured that there is no such number with this property.
\ \\

We close our final remarks with the inference that the truth of Conjecture~\ref{distinct_coprime} is a necessary condition for the truth of the odd $k$-perfect number conjecture. If there were any coprime $x,y\in\mN$ with $x\ne y$ and ${\displaystyle\frac{\sigma(x)}{x}=\frac{\sigma(y)}{y}}$ then an odd multiply perfect number would also have to exist.

\vspace*{1mm}
\begin{coro}
The inequality of the function values of  \ ${\displaystyle\frac{\sigma(x)}{x}}$ \ for any pair of different\vspace*{1mm}\lb coprime natural numbers is a necessary condition for the non-existence of an odd multiply perfect number.
\end{coro}
\vspace*{1mm}
\begin{proof}
The assertion is equivalent to that of Lemma~\refs{odd_coprime}.
\end{proof}

\ \\


\subsection*{Contact}
marioziller@arcor.de
\ \\\\\\
\addcontentsline{toc}{section}{\phj References\phj}
\bibliographystyle{amsplain}
\bibliography{References}     

\end{document}